\documentclass[11pt]{article}

\usepackage{amssymb,amsmath,bm}
\usepackage{amsthm}

\usepackage{textcomp}
\usepackage{enumerate}      
\usepackage{graphicx}        
\usepackage{caption}

\usepackage{tabu}

\usepackage[export]{adjustbox}

\usepackage{mathrsfs}

\usepackage{url}

\usepackage{array}
\newcommand{\PreserveBackslash}[1]{\let\temp=\\#1\let\\=\temp}
\newcolumntype{C}[1]{>{\PreserveBackslash\centering}p{#1}}
\newcolumntype{R}[1]{>{\PreserveBackslash\raggedleft}p{#1}}
\newcolumntype{L}[1]{>{\PreserveBackslash\raggedright}p{#1}}

\renewcommand{\setminus}{{\smallsetminus}}

\newcommand{\cp}[1]{\vcenter{\hbox{#1}}}

\usepackage{amssymb,amsmath,bm}
\usepackage{amsthm}
\usepackage{hyperref}
\usepackage{mathrsfs}
\usepackage{textcomp}
\usepackage{enumerate}
\usepackage{graphicx}
\usepackage{tikz}
\usepackage{verbatim}

\newtheorem{theorem}{Theorem}[section]
\newtheorem{lemma}[theorem]{Lemma}
\newtheorem{proposition}[theorem]{Proposition}
\newtheorem{definition}[theorem]{Definition}
\newtheorem{corollary}[theorem]{Corollary}
\newtheorem{conjecture}[theorem]{Conjecture}

\theoremstyle{remark}
\newtheorem{remark}[theorem]{Remark}

\theoremstyle{remark}
\newtheorem{example}[theorem]{Example}

\numberwithin{equation}{section}

\begin{document}
\title{\bf Discrete Fourier transforms, quantum $6j$-symbols and deeply truncated tetrahedra}

\author{Giulio Belletti and Tian Yang}

\date{}

\maketitle

\begin{abstract}  The asymptotic behavior of quantum $6j$-symbols is closely related to the volume of truncated hyperideal tetrahedra\,\cite{C}, and plays a central  role in understanding the asymptotics of the Turaev-Viro invariants of $3$-manifolds. In this paper, we propose a conjecture relating the asymptotics of the discrete Fourier transforms of quantum $6j$-symbols on one hand, and the volume of deeply truncated tetrahedra of various types on the other. As supporting evidence, we prove the conjecture in the case that the dihedral angles are sufficiently small, and provide numerical calculations in the case that the dihedral angles are relatively big. A key observation is a relationship between quantum $6j$-symbols and the co-volume function of deeply truncated tetrahedra, which is of interest in its own right and has important applications\,\cite{WY2, Ya}. More ambitiously,  we extend the conjecture to the discrete Fourier transforms of the Yokota invariants of planar graphs and volume of deeply truncated polyhedra, and provide supporting evidence.
\end{abstract}


\section{Introduction}

Quantum $6j$-symbols are the main building blocks of the Turaev-Viro invariants of $3$-manifolds\,\cite{TV}; the asymptotic behavior of the former plays a central role in understanding the asymptotics of the latter\,\cite{CY, BEL, BDKY}. In \cite{C}, it is proved that the exponential growth rate of the quantum $6j$-symbols is closely related to the volume of truncated hyperideal tetrahedra. In this paper, we aim to study the discrete Fourier transforms of quantum $6j$-symbols and the relationship between their asymptotic behavior and the volume of deeply truncated tetrahedra of various types.

To be precise, let $(I,J)$ be a partition of $\{1,\dots,6\},$ and let $\Delta$ be a deeply truncated tetrahedron of type $(I,J),$ ie., $\{e_i\}_{i\in I}$ is the set of edges of deep truncation. (See Section \ref{tetra} for more details.)

For a $6$-tuple $(b_I, a_J)=((b_i)_{i\in I},(a_j)_{j\in J})$ of integers in $\{0,\dots, r-2\},$   let $\mathrm{\widehat {Y}}_r\big(b_I; a_J\big)$ be the discrete Fourier transform of the Yokota invariant of the trivalent graph $\cp{\includegraphics[width=0.5cm]{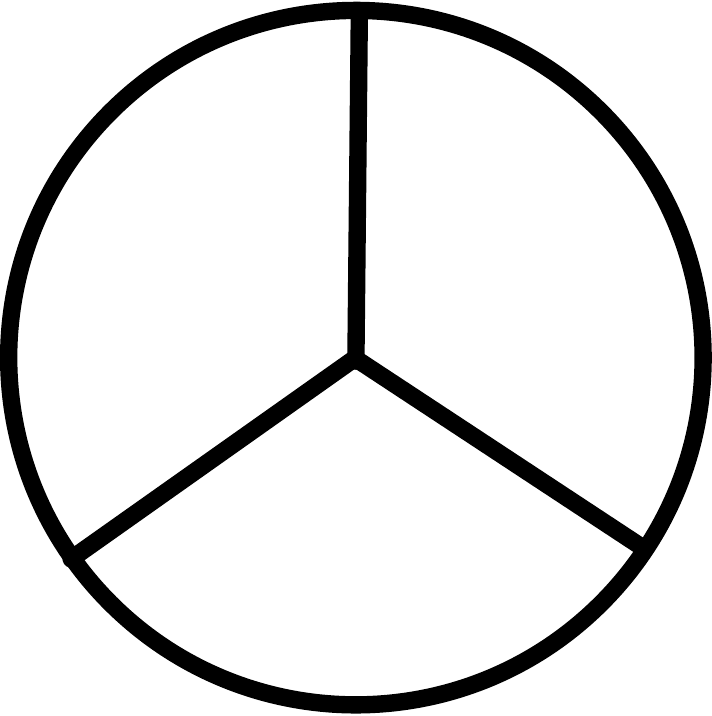}}$ with respect to $(b_I,a_J),$ ie.,
$$\mathrm{ \widehat {Y}}_r\big(b_I; a_J\big)=\sum_{(a_i)_{i\in I}}\prod_{i\in I} \mathrm{H}(a_i,b_i)\bigg|\begin{matrix}
a_1  & a_2 & a_3\\
   a_4 & a_5 & a_6
  \end{matrix}\bigg|^2$$
where the sum is over all multi-integers $(a_i)_{i\in I}$ in $\{0,\dots, r-2\}$ so that the triples $(a_1, a_2, a_3),$ $(a_1, a_5, a_6),$ $(a_2, a_4, a_6)$ and $(a_3, a_4, a_5)$  are $r$-admissible, 
$$\mathrm{H}(a_i,b_i)=(-1)^{a_i+b_i}\frac{q^{(a_i+1)(b_i+1)}-q^{-(a_i+1)(b_i+1)}}{q-q^{-1}},$$
and $\bigg|\begin{matrix} a_{1} & a_{2}  & a_{3} \\  a_{4}  &  a_{5}  &  a_{6}  \end{matrix} \bigg|$
is the quantum $6j$-symbol of the $6$-tuple $(a_{1},\dots,{a_{6}}).$ 
(See Sections \ref{6jsymbols}, \ref{yok} and \ref{pdft} for more details.)

\begin{conjecture}\label{conj} Suppose $\Delta(\theta_I;\theta_J)$ is a deeply truncated tetrahedron  of type $(I,J)$ with $\theta_I=\{\theta_i\}_{i\in I}$ the set of dihedral angles at the edges of deep truncation and  $\theta_J=\{\theta_j\}_{j\in J}$ the set of dihedral angles at the regular edges. Let $\{(b_I^{(r)}, a_J^{(r)})\}$ be a sequence of $6$-tuples  with
$$\theta_i=\Big|\pi-\lim_{r\to \infty}\frac{2\pi b_i^{(r)}}{r}\Big|$$
for $i\in I,$ and 
$$\theta_j=\Big|\pi-\lim_{r\to \infty}\frac{2\pi a_j^{(r)}}{r}\Big|$$ 
for $j\in J.$ Then evaluated at the root of unity $q=e^{\frac{2\pi \sqrt{-1}}{r}}$ and as $r$ varies over all positive odd integers,
$$\lim_{r\to \infty}\frac{2\pi}{r}\log \mathrm{\widehat {Y}}_r\big(b_I^{(r)};a_J^{(r)}\big)=2\rm{Vol}(\Delta(\theta_I;\theta_J)).$$
\end{conjecture}

As a convincing supporting evidence, we prove the following main result of this paper.

\begin{theorem}\label{main}
For any $(I,J),$ there exists an $\epsilon>0$ such that if all the dihedral angles of  $\Delta(\theta_I;\theta_J)$ are  less than $\epsilon,$ then Conjecture \ref{conj} is true.
\end{theorem} 

Furthermore, we provide additional numerical evidence for Conjecture \ref{conj} with relatively big dihedral angles in Section \ref{appendix}.

An analogous Volume Conjecture for deeply truncated tetrahedra with one edge of deep truncation was suggested (and has been verified numerically in a handful of cases) in \cite[Section 5, Conjecture 3]{KM2}. Notice however that the quantities inside the logarithm in Conjecture \ref{conj} and in \cite{KM2}, albeit similar-looking, are different.
\\

We also propose a more ambitious conjecture for the discrete Fourier transforms of the Yokota invariants of planar graphs which are the $1$-skeleton of deeply truncated polyhedra. (See Sections \ref{dtp}, \ref{yok} and \ref{pdft} for more details.)

\begin{conjecture}\label{Ambconj}
 Let $\Gamma\subset S^3$ be a planar graph and let $P$ be a deeply truncated polyhedron with $1$-skeleton $\Gamma.$ Let $\{\theta_i\}_{i\in I}$ be the angles at the edges of deep truncation and let $\{\theta_j\}_{j\in J}$ be the dihedral angles of $P$ at the regular edges. Let $\{(b^{(r)}_I,a^{(r)}_J)\}$ be a sequence of  colorings of $\Gamma$ such that
 $$\theta_i=\Big|\pi-\lim_{r\to \infty}\frac{2\pi b_i^{(r)}}{r}\Big|$$
for $i\in I$ and 
$$\theta_j=\Big|\pi-\lim_{r\to \infty}\frac{2\pi a_j^{(r)}}{r}\Big|$$ 
for $j\in J,$ and let  $\mathrm{\widehat {Y}}_r\big(\Gamma, b_I^{(r)};a_J^{(r)}\big)$ be the discrete Fourier transform of the Yokota invariant of $\Gamma$ with respect to $(b^{(r)}_I,a^{(r)}_J).$
Then as $r$ varies over all positive odd integers,
$$\lim_{r\to \infty}\frac{2\pi}{r}\log \mathrm{\widehat {Y}}_r\big(\Gamma,b_I^{(r)}; a_J^{(r)}\big)=2\rm{Vol}(P).$$

\end{conjecture}

\begin{remark} By using the same techniques as in the proof of Theorem \ref{main}, one can prove that Conjecture \ref{Ambconj} is true for all the graphs  obtained from $\cp{\includegraphics[width=0.5cm]{Yokota}}$ by doing a sequence of the following triangle moves $\cp{\includegraphics[width=1.5cm]{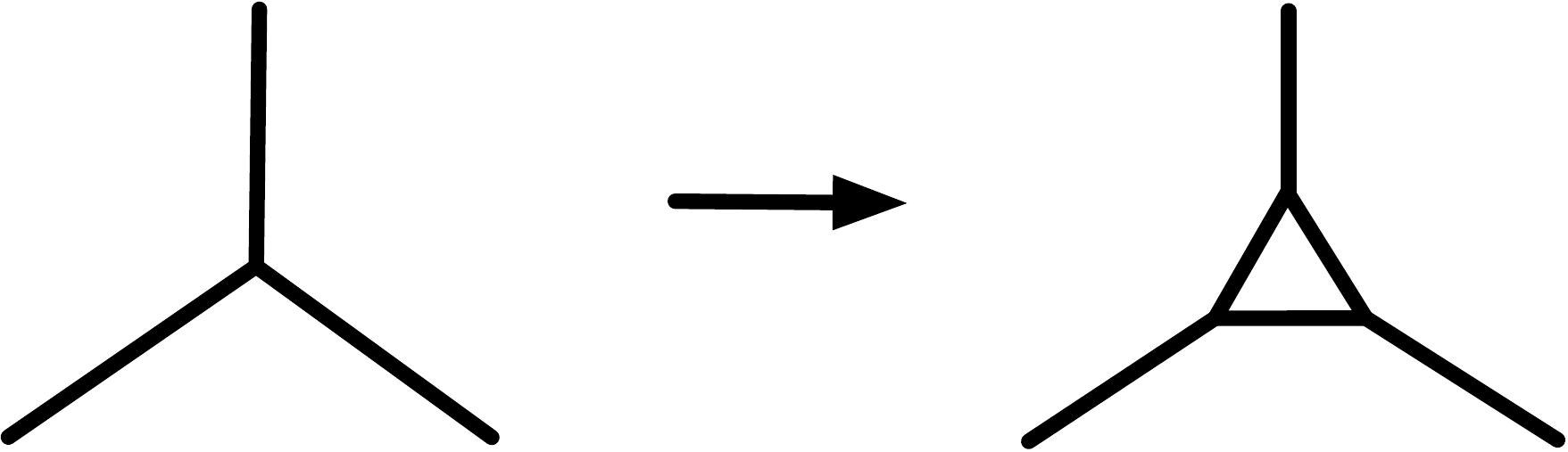}}$ and with sufficiently small dihedral angles, providing supporting evidences to Conjecture \ref{Ambconj}.
\end{remark}

\bigskip

\noindent\textbf{Outline of the proof of Theorem \ref{main}.}  
We follow the guideline of Ohtsuki's method. In Proposition \ref{computation}, we compute the discrete Fourier transform of the Yokota invariants of  the graph $\cp{\includegraphics[width=0.5cm]{Yokota}},$ writing them as a sum of values of a holomorphic function $f_r$ at integer points. The function $f_r$ comes from Faddeev's quantum dilogarithm function. Using Poisson Summation Formula, we in Proposition \ref{Poisson} write the invariants as a sum of the Fourier transforms of $f_r$ computed  in Propositions \ref{4.2}. In Proposition \ref{crit} we show that the critical value of the functions in the leading Fourier transforms has real part the volume of the deeply truncated tetrahedron. The key observation is Theorem \ref{co-vol} relating the asymptotics of quantum $6j$-symbols with  the co-volume function of deeply truncated tetrahedra, which is of interest in its own right. Then we estimate the leading Fourier transforms in Sections \ref{leading} using the Saddle Point Method (Proposition \ref{saddle}). Finally, we estimate the non-leading Fourier transforms and the error term respectively in Sections \ref{ot} and \ref{ee} showing that they are neglectable, and prove Theorem \ref{main} in Section \ref{pf}. We discuss the result in as general a case as possible, and only in the last section we use the assumption that the dihedral angles are sufficiently small.
\\

\noindent\textbf{Acknowledgments.}  The second author would like to thank Francis Bonahon, Zhengwei Liu, Feng Luo and Ka Ho Wong for helpful discussions. The second author is partially supported by NSF Grant DMS-1812008.


\section{Preliminaries}

\subsection{Deeply truncated tetrahedron and co-volume function}\label{tetra}

\begin{definition} A  \emph{ deeply truncated tetrahedron} is a compact hyperbolic polyhedron with faces $H_1,$ $H_2,$ $H_3,$ $H_4,$ $T_1,$ $T_2,$ $T_3,$ and $T_4$ such that
\begin{enumerate}[(1)]
\item For each  $i\in \{1,2,3,4\},$ $T_i\cap H_i=\emptyset.$

\item For each $\{i,j\} \subset \{1,2,3,4\},$ $T_i\cap H_j\neq\emptyset,$ and the dihedral angle between them is always $\frac{\pi}{2}.$

\item For each $\{i,j\} \subset \{1,2,3,4\},$ either $T_i\cap T_j\neq\emptyset$ or $H_i\cap H_j\neq\emptyset,$ but not both.
\end{enumerate}
\end{definition}

From the definition, we see that each face $H_i$ or $T_i$ is one of the following four types: (1) a hyperbolic triangle, (2) a hyperbolic quadrilateral with two right angles, (3) a hyperbolic pentagon with four right angles and (4) a hyperbolic hexagon with six right angles.

We only consider the intersection of $H_i$ and $H_j$ or the intersection of $T_i$ and $T_j$ as the \emph{edge} of the deeply truncated tetrahedron; therefore there are in total six edges. We call an edge between $H_i$ an $H_j$ a \emph{regular edge} and an edge between $T_i$ and $T_j$ an \emph{edge of deep truncation}.

A truncated hyperideal tetrahedron is one example of deeply truncated tetrahedra, where the $H_i$'s are the hexagonal faces, and $T_i$'s are the triangles from truncations (see Figure \ref{hyperideal}). For a truncated hyperideal tetrahedron, there are six regular edges and no edge of deep truncation. The name ``edge of deep truncation'' comes from the idea that if the truncations are ``deep'' enough, then two triangles $T_i$ and $T_j$ coming from truncations may intersect. That is also why the tetrahedron is called ``deeply truncated''. Deeply truncated tetrahedra were first studied by Kolpakov and Murakami in \cite{KM}.

\begin{figure}[htbp]
\centering
\includegraphics[scale=0.3]{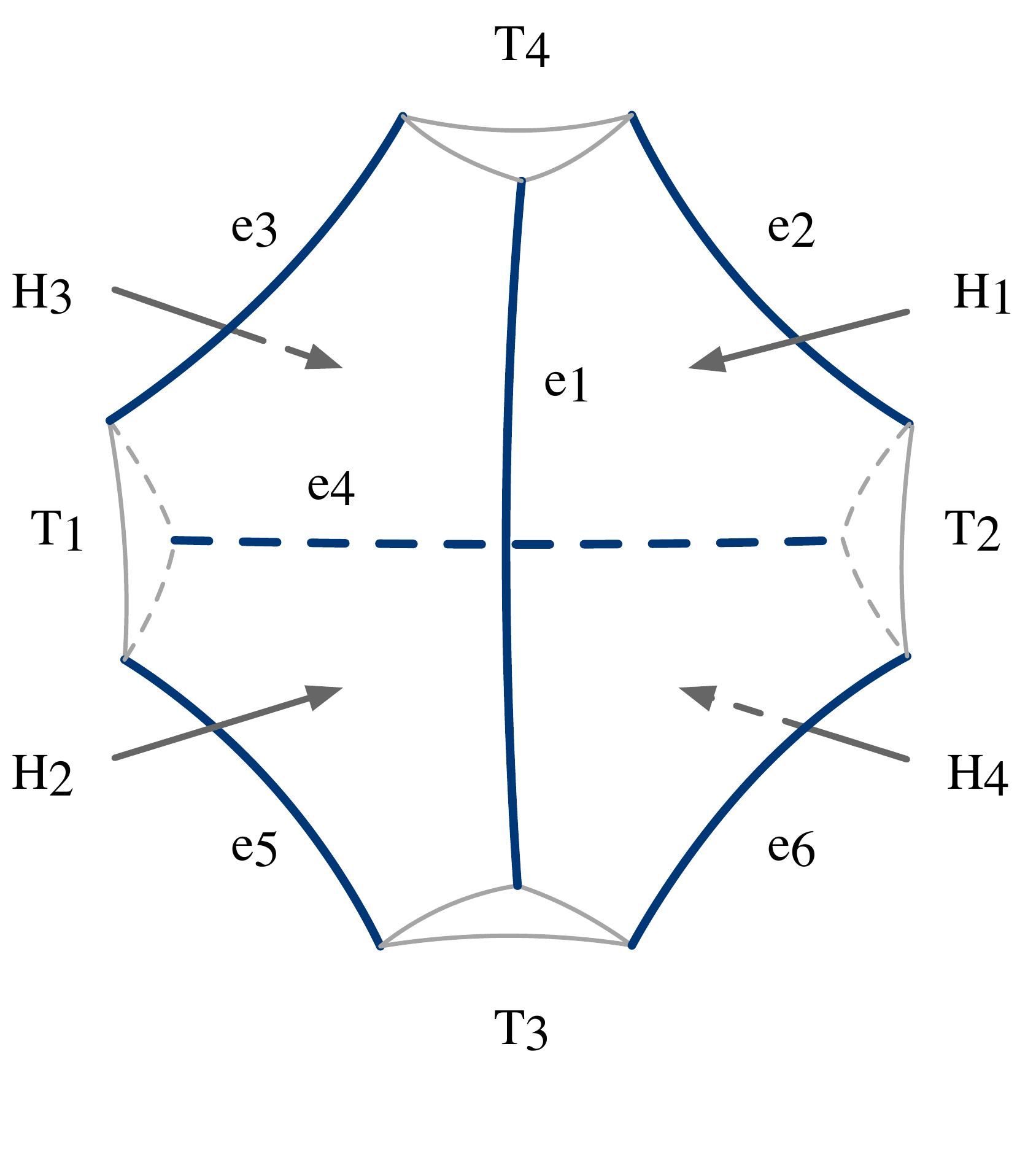}
\caption{Truncated hyperideal tetrahedron}
\label{hyperideal}
\end{figure}

 Let $(I,J)$ be a partition of $\{1,\dots,6\}.$ A deeply truncated tetrahedron $\Delta$ is of type $(I,J)$ if $\{e_i\}_{i\in I}$ is the set of edges of deep truncation. Up to permutation of indices and interchange the role of $H_i$'s and $T_i$'s, all the types of deeply truncated tetrahedra besides the truncated hyperideal tetrahedron are listed in Figure \ref{deep}.

\begin{figure}[htbp]
\centering
\includegraphics[scale=0.25]{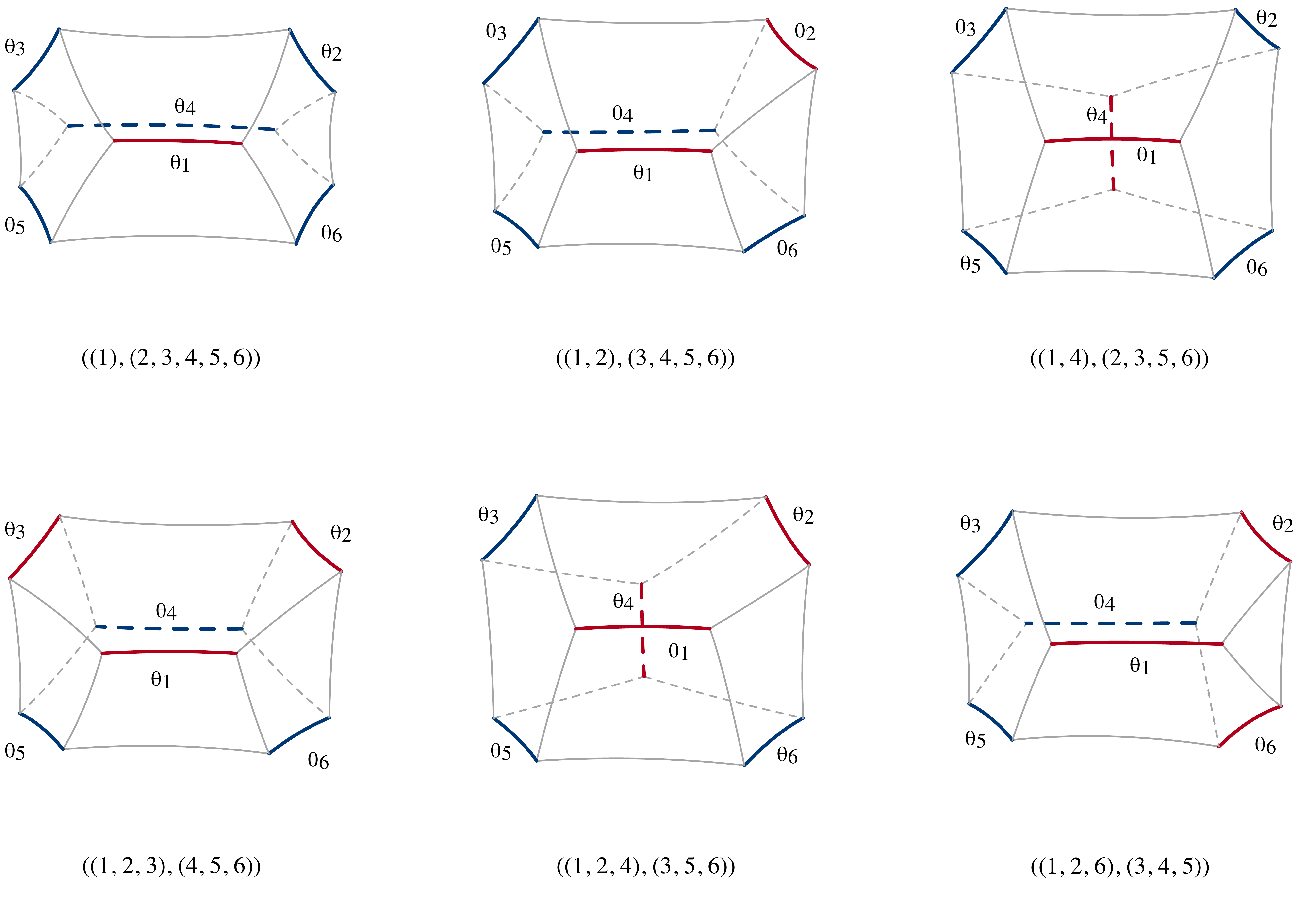}
\caption{Deeply truncated tetrahedra: The red edges are the edges of deep truncation and the blue edges are the regular edges. The dihedral angles at the grey edges are all right angles.}
\label{deep}
\end{figure}

Let $\Delta=\Delta(\theta_I,\theta_J)$ be the deeply truncated tetrahedron of type $(I,J);$ let $\{\theta\}_{i\in I}$ be the  dihedral angles at the edges of deep truncation and let $\{\theta_j\}_{j\in J}$ be the dihedral angles at the regular edges. Let $\{l_i\}_{i\in I}$ be the lengths of the edges of deep truncation. Let $c_i=\cosh l_i$ for $i\in I,$ and let $c_j=\cos\theta_j$ for $j\in J.$ If without loss of generality we assume that $1\in I,$ then by the hyperbolic cosine law of polygons of various types (see \cite[Appendix A]{GL2}) and a direct computation, we have
\begin{equation}\label{cos1}
\cos\theta_1=\frac{c_4+c_2c_6+c_3c_5+(c_2c_5+c_3c_6)c_1-c_4c_1^2}{\sqrt{-1+c_1^2+c_2^2+c_3^2+2c_1c_2c_3}\sqrt{-1+c_1^2+c_5^2+c_6^2+2c_1c_5c_6}}.
\end{equation}
If we consider the following Gram matrix 
\begin{equation*}
\begin{split}
G=\left[
\begin{array}{cccc}
1 & -c_1 & -c_2 & -c_6\\
-c_1& 1 & -c_3 & -c_5\\
-c_2 & -c_3 & 1 & -c_4 \\
-c_6 & -c_5 & -c_4  & 1 \\
 \end{array}\right],
 \end{split}
\end{equation*}
 then (\ref{cos1}) can be written as
\begin{equation}\label{cos}
\cos\theta_1=\frac{G_{34}}{\sqrt{G_{33}G_{44}}}.
\end{equation}

A deeply truncated tetrahedron $\Delta$ of type $(I,J)$ is up to isometry determined by $\{l_i\}_{i\in I}$ and $\{\theta_j\}_{j\in J}.$  In particular, if we fix $\{\theta_j\}_{j\in J},$ then the dihedral angles $\{\theta_i\}_{i\in I}$ at the edges of deep truncation are functions of $\{l_i\}_{i\in I}$ and as a consequence the volume of $\Delta$ is a function of $\{l_i\}_{i\in I}.$

\begin{definition}[\cite{Luo, LY}]\label{cov}
For a fixed $\{\theta_j\}_{j\in J},$ let $\mathrm{Vol}$ and $\{\theta_i\}_{i\in I}$ respectively be the volume of $\Delta$ and the dihedral angles at the edges of deep truncation as functions of $l_I=(l_i)_{i\in I},$ and consider the following co-volume function $\mathrm{Cov}$ defined by
$$\mathrm{Cov}(l_I)=\mathrm{Vol}+\frac{1}{2}\sum_{i\in I}\theta_i\cdot l_i .$$
\end{definition}

The key property of the co-volume function is the following

\begin{lemma}\label{Sch} For $i\in I,$
$$\frac{\partial \mathrm{Cov}}{\partial l_i}=\frac{\theta_i}{2}.$$
\end{lemma}

\begin{proof} By the Schl\"afli formula, we have
$$\frac{\partial \mathrm{Vol}}{\partial \theta_i}=-\frac{l_i}{2}.$$
Then by the chain rule and the product rule, we have
\begin{equation*}
\begin{split}
\frac{\partial \mathrm{Cov}}{\partial l_i}=&\sum_{k\in I}\frac{\partial \mathrm{Vol}}{\partial\theta_k}\cdot\frac{\partial \theta_k}{\partial l_i}+\frac{1}{2}\sum_{k\in I}\frac{\partial}{\partial l_i}\Big(\theta_k\cdot l_k\Big)\\
=&-\sum_{k\in I}\frac{l_k}{2}\cdot \frac{\partial \theta_k}{\partial l_i}+\frac{1}{2}\sum_{k\in I}\cdot\frac{\partial \theta_k}{\partial l_i}\cdot l_k+\frac{\theta_i}{2}=\frac{\theta_i}{2}.
\end{split}
\end{equation*}
\end{proof}

Finally we include a sufficient condition to determine whether a deeply truncated tetrahedron with specified parameters exists.

\begin{proposition}\label{prop:criterion}
Let $(I,J)$ be a partition of $\{1,\dots,6\},$ $\{l_i\}_{i\in I}$ be positive real numbers and $\{\theta_j\}_{j\in J}$ be real numbers in $[0,\pi].$ Let $c_i=\cosh l_i$ for $i\in I$ and $c_j=\cos\theta_j$ for $j\in J,$ and let  \begin{displaymath}
 G=\left[
\begin{array}{cccc}
1 & -c_1 & -c_2 & -c_6\\
-c_1& 1 & -c_3 & -c_5\\
-c_2 & -c_3 & 1 & -c_4 \\
-c_6 & -c_5 & -c_4  & 1 \\
 \end{array}\right]
\end{displaymath}
be the Gram matrix defined as above.
If
\begin{enumerate}[(1)]
 \item $\textrm{Sign}(G)=(3,1),$
 \item $G_{st}>0$ for $s\neq t$ and
 \item $G_{ss}<0,$
\end{enumerate}
then there exists a deeply truncated tetrahedron of type $(I,J)$ with lengths $\{l_i\}_{i\in I}$ of the edges of deep truncation and 
dihedral angles $\{\theta_j\}_{j\in J}$ at the regular edges.
\end{proposition}

\begin{proof}
 The first part of the proof in \cite[Theorem 3.2]{U} works verbatim in the case of deeply truncated tetrahedra. Notice that we require that a deeply truncated tetrahedron has hyperideal vertices, while [Ushijima, Theorem 3.2] does not; this accounts for the extra condition $(3)$ (see \cite[Remark 3 after Theorem 3.2]{U}).
\end{proof}

\begin{remark}
The conditions in \cite[Theorem 3.2]{U} are necessary and sufficient. However, if the pair $(s,t)$ gives entry of the Gram matrix corresponding to an edge of deep truncation, then the condition $G_{st}>0$ would imply that the dihedral angle at this edge is less than $\frac{\pi}{2},$ which is not always the case. Hence the conditions in Proposition \ref{prop:criterion} are not necessary in general.
\end{remark}

\subsection{Deeply truncated polyhedra}\label{dtp}

We extend the definition of a deeply truncated tetrahedron to polyhedra of any combinatorial type, which are the objects involved in Conjecture \ref{Ambconj}.

Let $\Gamma\subset S^3$ be a polyhedral graph (that is, a graph that is the $1$-skeleton of a polyhedron), and let $V$ be its set of vertices and $F$ its set of faces.
\begin{definition}
A \emph{deeply truncated polyhedron} with $1$-skeleton $\Gamma$ is a compact hyperbolic polyhedron $P\subset \mathbb{H}^3$ with faces $\{T_v\}_{v\in V}\cup \{H_f\}_{f\in F}$ such that:

\begin{enumerate}[(1)]
 \item $T_v\cap H_f\neq \varnothing$ if and only if $v\in f,$
 \item if $T_v\cap H_f\neq \varnothing,$ then they intersect at a right angle, and
 \item for every edge $e$ of $\Gamma$ with endpoints $v_1,v_2$ and adjacent to faces $f_1,f_2,$ exactly one of $T_{v_1}\cap T_{v_2}$ and $H_{f_1}\cap H_{f_2}$ is non-empty.
\end{enumerate}
The edge $e$ of $\Gamma$  is a \emph{regular edge} if $H_{f_1}\cap H_{f_2}\neq \varnothing;$  otherwise it is an \emph{edge of deep truncation}. In either case, the \emph{dihedral angle} of $P$ at $e$ is the angle between the two intersecting faces. 
\end{definition}

\begin{remark}
 Notice that $P$ by itself is simply a compact hyperbolic polyhedron; to obtain a deeply truncated polyhedron with $1$-skeleton $\Gamma$ we need the additional information of the partition of its faces according to vertices and faces of $\Gamma.$ In particular, $P$ as a simplicial complex in $\mathbb{H}^3$ does \emph{not} have $1$-skeleton $\Gamma;$ by definition $\Gamma$ is the $1$-skeleton of $P$ with the additional information. 
\end{remark}

\begin{example}
A truncated  hyperideal  polyhedron $P$ with $1$-skeleton $\Gamma$\,\cite{BB} is a deeply truncated polyhedron with $1$-skeleton $\Gamma.$ The faces of $P$ give the set $H_F$ while the faces dual to the vertices of $P$ give the set $T_V.$ In this case all edges are regular.
\end{example}

\begin{remark}
 The statement of Proposition \ref{Sch} is true for deeply truncated polyhedra as well; the same proof applies verbatim.
\end{remark}


\subsection{Quantum \texorpdfstring{$6j$}{6j}-symbols}\label{6jsymbols}

Let $r$ be an odd integer and $q$ be an $r$-th root of unity. For the context of this paper we are only interested in the case $q=e^{\frac{2\pi\sqrt{-1}}{r}},$ but the definitions and results in this section work with any choice of $q.$

As is customary we define $[n]=\frac{q^n-q^{-n}}{q-q^{-1}},$ $\{n\}=q^n-q^{-n}$ and the quantum factorial 
$$[n]!=\prod_{k=1}^n[k].$$

A triple $(a_1,a_2,a_3)$ of integers in $\{0,\dots,r-2\}$ is \emph{$r$-admissible} if 
\begin{enumerate}[(1)]
\item  $a_i+a_j-a_k\geqslant 0$ for $\{i,j,k\}=\{1,2,3\}.$
\item $a_1+a_2+a_3\leqslant 2(r-2),$ 
\item $a_1+a_2+a_3$ is even.
\end{enumerate}

For an $r$-admissible triple $(a_1,a_2,a_3),$ define 
$$\Delta(a_1,a_2,a_3)=\sqrt{\frac{[\frac{a_1+a_2-a_3}{2}]![\frac{a_2+a_3-a_1}{2}]![\frac{a_3+a_1-a_2}{2}]!}{[\frac{a_1+a_2+a_3}{2}+1]!}}$$
with the convention that $\sqrt{x}=\sqrt{|x|}\sqrt{-1}$ when the real number $x$ is negative.

A  6-tuple $(a_1,\dots,a_6)$ is \emph{$r$-admissible} if the triples $(a_1,a_2,a_3),$ $(a_1,a_5,a_6),$ $(a_2,a_4,a_6)$ and $(a_3,a_4,a_5)$ are $r$-admissible.

\begin{definition}
The \emph{quantum $6j$-symbol} of an $r$-admissible 6-tuple $(a_1,\dots,a_6)$ is 
\begin{multline*}
\bigg|\begin{matrix} a_1 & a_2 & a_3 \\ a_4 & a_5 & a_6 \end{matrix} \bigg|
= \sqrt{-1}^{-\sum_{i=1}^6a_i}\Delta(a_1,a_2,a_3)\Delta(a_1,a_5,a_6)\Delta(a_2,a_4,a_6)\Delta(a_3,a_4,a_5)\\
\sum_{k=\max \{T_1, T_2, T_3, T_4\}}^{\min\{ Q_1,Q_2,Q_3\}}\frac{(-1)^k[k+1]!}{[k-T_1]![k-T_2]![k-T_3]![k-T_4]![Q_1-k]![Q_2-k]![Q_3-k]!},
\end{multline*}
where $T_1=\frac{a_1+a_2+a_3}{2},$ $T_2=\frac{a_1+a_5+a_6}{2},$ $T_3=\frac{a_2+a_4+a_6}{2}$ and $T_4=\frac{a_3+a_4+a_5}{2},$ $Q_1=\frac{a_1+a_2+a_4+a_5}{2},$ $Q_2=\frac{a_1+a_3+a_4+a_6}{2}$ and $Q_3=\frac{a_2+a_3+a_5+a_6}{2}.$
\end{definition}

Closely related, a triple $(\alpha_1,\alpha_2,\alpha_3)\in [0,2\pi]^3$ is \emph{admissible} if 
\begin{enumerate}[(1)]
\item $\alpha_i+\alpha_j-\alpha_k\geqslant 0$ for $\{i,j,k\}=\{1,2,3\},$
\item $\alpha_i+\alpha_j+\alpha_k\leqslant 4\pi.$
\end{enumerate}
A $6$-tuple $(\alpha_1,\dots,\alpha_6)\in [0,2\pi]^6$ is \emph{admissible} if the triples $\{1,2,3\},$ $\{1,5,6\},$ $\{2,4,6\}$ and $\{3,4,5\}$ are admissible.


\subsection{The Yokota invariant}\label{yok}

In this section we recall the definition of the Yokota invariant, first introduced in \cite{Y}. It is an invariant that extends the Kauffman bracket for trivalent graphs to the case of graphs with vertices of any valence. For the sake of simplicity we only deal with the case of planar graphs with no $1$- or $2$-valent vertices; the general case of framed graphs in closed oriented manifolds is not conceptually more complex.

Let $\Gamma\subset S^3$ be a trivalent planar graph, $a_I$ be a coloring of its edges with elements in $\{0,\dots,r-2\},$ and denote with $\langle \Gamma,a_I\rangle$ the \emph{Kauffman bracket} evaluated at the $r$-th root of unity $q=e^{\frac{2\pi\sqrt{-1}}{r}}$ of the pair $\left(\Gamma,a_I\right)$ (see for example \cite[Section 9]{KL} for a definition). 
 
We say that the coloring $a_I$ is \emph{$r$-admissible} if, whenever $i,j,k\in I$ are the indices of the edges of $\Gamma$ sharing a vertex, the triple $(a_i,a_j,a_k)$ is $r$-admissible. 

\begin{definition}
 A \emph{desingularization} of a planar graph $\Gamma$ with no $1$- and $2$-valent vertices is a graph $\Gamma'$ that coincides with $\Gamma$ outside of a neighborhood of the vertices of $\Gamma,$ and in a neighborhood of each vertex is a planar trivalent tree, as in Figure \ref{fig:desing}.
\end{definition}

\begin{figure}
\centering
 \begin{minipage}{.4\textwidth}\centering \begin{tikzpicture}[scale=0.5]
\centering
\draw[thick] (3,3)--(3,-3);
\draw[thick] (0.5,2)--(5.5,-2);
\draw[thick] (5.5,2)--(0.5,-2);
\end{tikzpicture}
\end{minipage}
$\xrightarrow{\hspace*{1cm}}$
\begin{minipage}{.4\textwidth}
\centering
  \begin{tikzpicture}[scale=0.4]
\centering
\draw[thick] (3,3)--(3,0);
\draw[thick] (3,2)--(3,-4);
\draw[thick] (0.5,2)--(3,0);
\draw[thick] (5.5,2)--(3,1);
\draw[thick](0.5,-3)--(3,-2);
\draw[thick] (3,-1)--(5.5,-3);
\end{tikzpicture}
\end{minipage}
\caption{Desingularization in a neigborhood of a vertex of valence $6$}\label{fig:desing}
\end{figure}
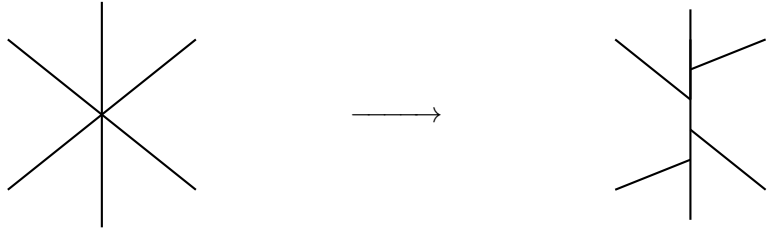

\begin{definition}
 Let $\Gamma\subset S^3$ be a planar graph with no $1$- and $2$-valent vertices and $\Gamma'$ be a desingularization of $\Gamma.$ Let $I$ be the set of edges of $\Gamma$ and $I'$ be the set of edges of $\Gamma'$ so that $I\subset I'$ in a natural way, and let $V'$ be the set of vertices of $\Gamma'.$ Then the $r$-th  \emph{Yokota invariant} of $\left(\Gamma,a_I\right)$ is 
 
 \begin{displaymath}
  \mathrm Y_r(\Gamma,a_I)= \sum_{a_{I'}}\frac{\prod_{i\in I'\setminus I} \langle \cp{\includegraphics[width=0.5cm]{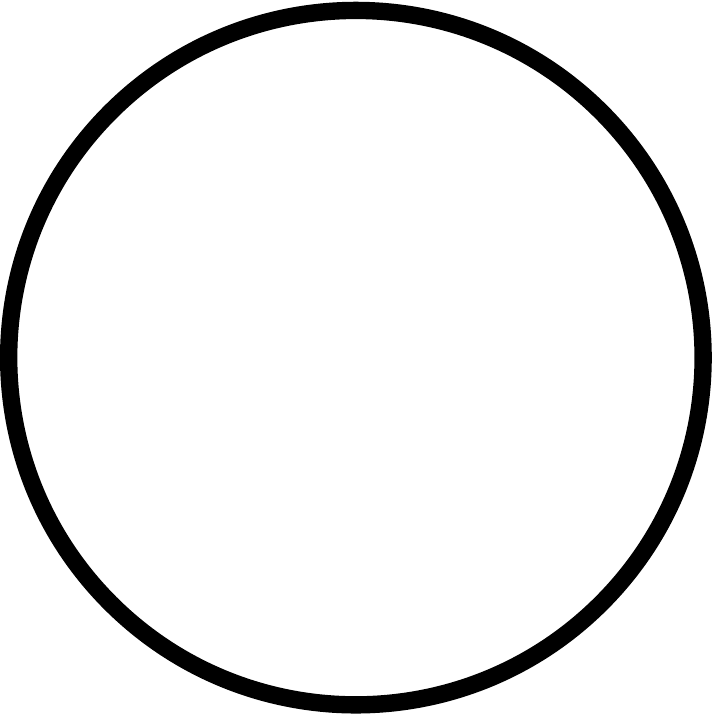}}, a_i\rangle}{\prod_{v\in V'} \langle \cp{\includegraphics[width=0.5cm]{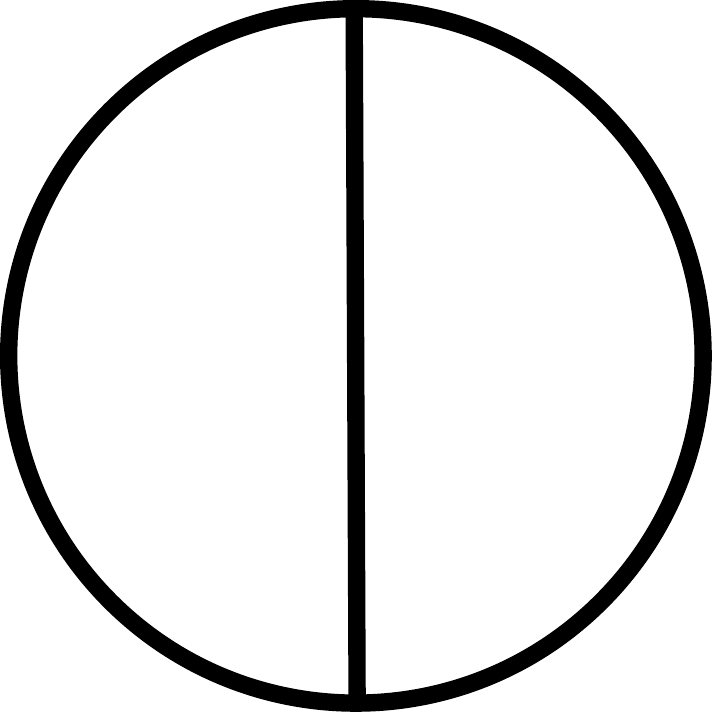}}, (a_{v_1},a_{v_2},a_{v_3})\rangle}\langle \Gamma',a_{I'}\rangle^2,
 \end{displaymath}
where the sum is over all $r$-admissible colorings $a_{I'}$ of $I'$ extending $a_I,$ and $(v_1,v_2,v_3)$ are the indices of the edges of $\Gamma'$ sharing $v.$ In particular, \begin{displaymath}
 \mathrm Y_r(\cp{\includegraphics[width=0.5cm]{Yokota}},(a_1,\dots, a_6))= \bigg|\begin{matrix} a_1 & a_2 & a_3 \\ a_4 & a_5 & a_6 \end{matrix} \bigg|^2.
 \end{displaymath}
\end{definition}

\begin{remark}
If $\Gamma$ has at least a vertex of valence greater than $3,$ then there are many different desingularizations; but $\mathrm Y_r(\Gamma,a_I)$ is independent of the choice of desingularization (\cite[Proposition 4.3]{Y}).
\end{remark}

\subsection{Discrete Fourier transforms}\label{pdft}

The discrete Fourier transforms of the Yokota invariants were introduced in \cite{BAR};  they were later shown to be a particular case of a general construction for modular tensor categories (see \cite[Section 1]{LIU} and references therein). They were used to prove the Turaev-Viro volume conjecture for a certain family of manifolds in \cite{BEL}, using a weak form of the Poisson Summation Formula; later \cite{WY2} established a strong version of the Poisson Summation Formula.

For simplicity we only deal with planar graphs in $S^3;$ and the definition could be extended to any graph in any closed, oriented $3$-manifold without any extra difficulty.

\begin{definition}
 Let $\Gamma$ be a planar graph in $S^3,$ $(I,J)$ be a partition of the edges of $\Gamma,$ and $(b_I, a_J)$ be a coloring of the edges of $\Gamma.$ The \emph{discrete Fourier transform} of the Yokota invariant of $\Gamma$ with respect to  $(b_I,a_J)$ is defined as 
 $$\mathrm{ \widehat {Y}}_r\big(\Gamma, b_I; a_J\big)=\sum_{(a_i)_{i\in I}}\prod_{i\in I} \mathrm{H}(a_i,b_i)\mathrm Y_r(\Gamma,a_I,a_J).$$
\end{definition}

\begin{proposition}[Duality of the Fourier transform of the $6j$-symbol]\label{prop:duality}
 Let $(I,J)$ be any partition of the edges of the tetrahedral graph. Then 
 \begin{displaymath}
  \mathrm{\widehat {Y}}_r(b_I;a_J)=\bigg(\frac{ r}{2\sin^2\big(\frac{2\pi}{r}\big)}\bigg)^{3-\lvert I\rvert} \mathrm{\widehat {Y}}_r(a_J;b_I).
 \end{displaymath}
\end{proposition}
\begin{proof}
The case of $J=\varnothing$ was proven in \cite{BAR}. The general case can be reduced to the case of $J=\varnothing$ using the fact that 
\begin{displaymath}
 \sum_{i=0}^{r-2}\sum_{j=0}^{r-2}\mathrm H(k,i)\mathrm H(j,l)= \frac{2\sin^2\big(\frac{2\pi}{r}\big)}{r} \delta_{kl}.
\end{displaymath}

\end{proof}

\subsection{Dilogarithm and Lobachevsky functions}

Let $\log:\mathbb C\setminus (-\infty, 0]\to\mathbb C$ be the standard logarithm function defined by
$$\log z=\log|z|+\sqrt{-1}\cdot\arg z$$
with $-\pi<\arg z<\pi.$
 
The dilogarithm function $\mathrm{Li}_2: \mathbb C\setminus (1,\infty)\to\mathbb C$ is defined by
$$\mathrm{Li}_2(z)=-\int_0^z\frac{\log (1-u)}{u}du$$
where the integral is along any path in $\mathbb C\setminus (1,\infty)$ connecting $0$ and $z,$ which is holomorphic in $\mathbb C\setminus [1,\infty)$ and continuous in $\mathbb C\setminus (1,\infty).$

The dilogarithm function satisfies the following properties (see eg. Zagier\,\cite{Z}).
\begin{enumerate}[(1)]
\item for any $z\in \mathbb{C}\setminus(1,\infty),$\begin{equation}\label{Li2}
\mathrm{Li}_2\Big(\frac{1}{z}\Big)=-\mathrm{Li}_2(z)-\frac{\pi^2}{6}-\frac{1}{2}\big(\log(-z)\big)^2.
\end{equation} 
\item In the unit disk $\big\{z\in\mathbb C\,\big|\,|z|<1\big\},$ 
\begin{equation}\label{Li1}
\mathrm{Li}_2(z)=\sum_{n=1}^\infty\frac{z^n}{n^2},
\end{equation}
\item On the unit circle $\big\{ z=e^{2\sqrt{-1}\theta}\,\big|\,0 \leqslant \theta\leqslant\pi\big\},$ 
\begin{equation}\label{dilogLob}
\mathrm{Li}_2(e^{2\sqrt{-1}\theta})=\frac{\pi^2}{6}+\theta(\theta-\pi)+2\sqrt{-1}\cdot\Lambda(\theta).
\end{equation}
\end{enumerate}
Here $\Lambda:\mathbb R\to\mathbb R$  is the Lobachevsky function defined by
$$\Lambda(\theta)=-\int_0^\theta\log|2\sin t|dt,$$
which is an odd function of period $\pi$ (see eg. Thurston's notes\,\cite[Chapter 7]{T}).



\subsection{Quantum dilogarithm functions}

The following variant of Faddeev's quantum dilogarithm functions\,\cite{F, FKV} will play a key role in the proof of the main result. 
Let $r\geqslant 3$ be an odd integer. Then the following contour integral
\begin{equation}
\varphi_r(z)=\frac{4\pi \sqrt{-1}}{r}\int_{\Omega}\frac{e^{(2z-\pi)x}}{4x \sinh (\pi x)\sinh (\frac{2\pi x}{r})}\ dx
\end{equation}
defines a holomorphic function on the domain $$\Big\{z\in \mathbb C \ \Big|\ -\frac{\pi}{r}<\mathrm{Re}z <\pi+\frac{\pi}{r}\Big\},$$  
  where the contour is
$$\Omega=\big(-\infty, -\epsilon\big]\cup \big\{z\in \mathbb C\ \big||z|=\epsilon, \mathrm{Im}z>0\big\}\cup \big[\epsilon,\infty\big),$$
for some $\epsilon\in(0,1).$
Note that the integrand has poles at $\sqrt{-1} n,$ $n\in\mathbb Z,$ and the choice of  $\Omega$ is to avoid the pole at $0.$
\\

The function $\varphi_r(z)$ satisfies the following fundamental properties; their proof can be found in \cite[Lemma 2.1]{WY}.
\begin{lemma}
\begin{enumerate}[(1)]
\item For $z\in\mathbb C$ with  $0<\mathrm{Re}z<\pi,$
\begin{equation}\label{fund}
1-e^{2 \sqrt{-1} z}=e^{\frac{r}{4\pi \sqrt{-1}}\Big(\varphi_r\big(z-\frac{\pi}{r}\big)-\varphi_r\big(z+\frac{\pi}{r}\big)\Big)}.
 \end{equation}
 
 \item For $z\in\mathbb C$ with  $-\frac{\pi}{r}<\mathrm{Re}z<\frac{\pi}{r},$
 \begin{equation}\label{f2}
1+e^{r\sqrt{-1}z}=e^{\frac{r}{4\pi \sqrt{-1}}\Big(\varphi_r(z)-\varphi_r\big(z+\pi\big)\Big)}.
\end{equation}
\end{enumerate}
\end{lemma}

Using (\ref{fund}) and (\ref{f2}), for $z\in\mathbb C$ with $\pi+\frac{2(n-1)\pi}{r}< \mathrm{Re}z< \pi+\frac{2n\pi}{r},$ we can define $\varphi_r(z)$  inductively by the relation
\begin{equation}\label{extension}
\prod_{k=1}^n\Big(1-e^{2 \sqrt{-1} \big(z-\frac{(2k-1)\pi}{r}\big)}\Big)=e^{\frac{r}{4\pi \sqrt{-1}}\Big(\varphi_r\big(z-\frac{2n\pi}{r}\big)-\varphi_r(z)\Big)},
\end{equation}
extending $\varphi_r(z)$ to a meromorphic function on $\mathbb C.$  The poles of $\varphi_r(z)$ have the form $(a+1)\pi+\frac{b\pi}{r}$ or $-a\pi-\frac{b\pi}{r}$ for all nonnegative integer $a$ and positive odd integer $b.$

Let $q=e^{\frac{2\pi \sqrt{-1}}{r}},$
and let $$(q)_n=\prod_{k=1}^n(1-q^{2k}).$$

\begin{lemma}\label{fact}
\begin{enumerate}[(1)]
\item For $0\leqslant n \leqslant r-2,$
\begin{equation}
(q)_n=e^{\frac{r}{4\pi \sqrt{-1}}\Big(\varphi_r\big(\frac{\pi}{r}\big)-\varphi_r\big(\frac{2\pi n}{r}+\frac{\pi}{r}\big)\Big)}.
\end{equation}
\item For $\frac{r-1}{2}\leqslant n \leqslant r-2,$
\begin{equation} 
(q)_n=2e^{\frac{r}{4\pi \sqrt{-1}}\Big(\varphi_r\big(\frac{\pi}{r}\big)-\varphi_r\big(\frac{2\pi n}{r}+\frac{\pi}{r}-\pi\big)\Big)}.
\end{equation}
\end{enumerate}
\end{lemma}

Let $\{n\}!=\prod_{k=1}^n\{k\}.$ Then
$$\{n\}!=(-1)^nq^{-\frac{n(n+1)}{2}}(q)_n,$$ and
as a consequence of Lemma \ref{fact}, we have

\begin{lemma}\label{factorial}
\begin{enumerate}[(1)]
\item For $0\leqslant n \leqslant r-2,$
\begin{equation}
\{n\}!=e^{\frac{r}{4\pi \sqrt{-1}}\Big(-2\pi\big(\frac{2\pi n}{r}\big)+\big(\frac{2\pi}{r}\big)^2(n^2+n)+\varphi_r\big(\frac{\pi}{r}\big)-\varphi_r\big(\frac{2\pi n}{r}+\frac{\pi}{r}\big)\Big)}.
\end{equation}
\item For $\frac{r-1}{2}\leqslant n \leqslant r-2,$
\begin{equation} \label{move}
\{n\}!=2e^{\frac{r}{4\pi \sqrt{-1}}\Big(-2\pi\big(\frac{2\pi n}{r}\big)+\big(\frac{2\pi }{r}\big)^2(n^2+n)+\varphi_r\big(\frac{\pi}{r}\big)-\varphi_r\big(\frac{2\pi n}{r}+\frac{\pi}{r}-\pi\big)\Big)}.
\end{equation}
\end{enumerate}
\end{lemma}

We consider (\ref{move}) because there are poles in $(\pi,2\pi),$ and to avoid the poles we move the variables to $(0,\pi)$ by subtracting $\pi.$

The function $\varphi_r(z)$ and the dilogarithm function are closely related as follows.

\begin{lemma}\label{converge}  \begin{enumerate}[(1)]
\item For every $z$ with $0<\mathrm{Re}z<\pi,$ 
\begin{equation}\label{conv1}
\varphi_r(z)=\mathrm{Li}_2(e^{2\sqrt{-1}z})+\frac{2\pi^2e^{2\sqrt{-1}z}}{3(1-e^{2\sqrt{-1}z})}\frac{1}{r^2}+O\Big(\frac{1}{r^4}\Big).
\end{equation}
\item For every $z$ with $0<\mathrm{Re}z<\pi,$ 
\begin{equation}\label{conv2}
\varphi_r'(z)=-2\sqrt{-1}\cdot\log(1-e^{2\sqrt{-1}z})+O\Big(\frac{1}{r^2}\Big).
\end{equation}
\item \cite[Formula (8)(9)]{O2}
$$\varphi_r\Big(\frac{\pi}{r}\Big)=\mathrm{Li}_2(1)+\frac{2\pi\sqrt{-1}}{r}\log\Big(\frac{r}{2}\Big)-\frac{\pi^2}{r}+O\Big(\frac{1}{r^2}\Big).$$
\end{enumerate}\end{lemma}


\section{The geometry of quantum \texorpdfstring{$6j$}{6j}-symbols}

\begin{definition} An $r$-admissible $6$-tuple $(a_1,\dots,a_6)$  is of the \emph{hyperideal type} if for $\{i,j,k\}=\{1,2,3\},$ $\{1,5,6\},$ $\{2,4,6\}$ and $\{3,4,5\},$
\begin{enumerate}[(1)]
\item $0\leqslant a_i+a_j-a_k<r-2,$
\item $r-2<a_i+a_j+a_k\leqslant 2(r-2),$
\item $a_i+a_j+a_k$ is even.
\end{enumerate}
\end{definition}

As a consequence of Lemma \ref{factorial} we have

\begin{proposition}\label{6jqd} The quantum $6j$-symbol at the root of unity $q=e^{\frac{2\pi \sqrt{-1}}{r}}$ can be computed as 
$$\bigg|
\begin{matrix}
        a_1 & a_2 & a_3 \\
        a_4 & a_5 & a_6 
      \end{matrix} \bigg|=\frac{\{1\}}{2}\sum_{k=\max\{T_1,T_2,T_3,T_4\}}^{\min\{Q_1,Q_2,Q_3,r-2\}}e^{\frac{r}{4\pi \sqrt{-1}}U_r\big(\frac{2\pi a_1}{r},\dots,\frac{2\pi a_6}{r},\frac{2\pi k}{r}\big)},$$
 where $U_r$ is defined as follows. If $(a_1,\dots,a_6)$ is of hyperideal type, then
\begin{equation}\label{termwithr}
\begin{split}
U_r(\alpha_1,\dots,\alpha_6,\xi)=&\pi^2-\Big(\frac{2\pi}{r}\Big)^2+\frac{1}{2}\sum_{i=1}^4\sum_{j=1}^3(\eta_j-\tau_i)^2-\frac{1}{2}\sum_{i=1}^4\Big(\tau_i+\frac{2\pi}{r}-\pi\Big)^2\\
&+\Big(\xi+\frac{2\pi}{r}-\pi\Big)^2-\sum_{i=1}^4(\xi-\tau_i)^2-\sum_{j=1}^3(\eta_j-\xi)^2\\
&-2\varphi_r\Big(\frac{\pi}{r}\Big)-\frac{1}{2}\sum_{i=1}^4\sum_{j=1}^3\varphi_r\Big(\eta_j-\tau_i+\frac{\pi}{r}\Big)+\frac{1}{2}\sum_{i=1}^4\varphi_r\Big(\tau_i-\pi+\frac{3\pi}{r}\Big)\\
&-\varphi_r\Big(\xi-\pi+\frac{3\pi}{r}\Big)+\sum_{i=1}^4\varphi_r\Big(\xi-\tau_i+\frac{\pi}{r}\Big)+\sum_{j=1}^3\varphi_r\Big(\eta_j-\xi+\frac{\pi}{r}\Big),\\
\end{split}
\end{equation}
where $\alpha_i=\frac{2\pi a_i}{r}$ for $i=1,\dots,6$ and $\xi=\frac{2\pi k}{r},$ $\tau_1=\frac{\alpha_1+\alpha_2+\alpha_3}{2},$ $\tau_2=\frac{\alpha_1+\alpha_5+\alpha_6}{2},$ $\tau_3=\frac{\alpha_2+\alpha_4+\alpha_6}{2}$ and $\tau_4=\frac{\alpha_3+\alpha_4+\alpha_5}{2},$ $\eta_1=\frac{\alpha_1+\alpha_2+\alpha_4+\alpha_5}{2},$ $\eta_2=\frac{\alpha_1+\alpha_3+\alpha_4+\alpha_6}{2}$ and $\eta_3=\frac{\alpha_2+\alpha_3+\alpha_5+\alpha_6}{2}.$
If $(a_1,\dots,a_6)$ is not of the hyperideal type, then $U_r$ will be changed according to Lemma \ref{factorial}.
\end{proposition}

\begin{definition} A $6$-tuple $(\alpha_1,\dots,\alpha_6)\in [0,2\pi]^6$ is of the \emph{hyperideal type} if
\begin{enumerate}[(1)]
\item $0\leqslant \alpha_i+\alpha_j-\alpha_k\leqslant 2\pi,$
\item $2\pi\leqslant \alpha_i+\alpha_j+\alpha_k\leqslant 4\pi.$
\end{enumerate}
\end{definition}

We notice that the six numbers $|\pi-\alpha_1|,\dots,|\pi-\alpha_6|$ are the dihedral angles of an ideal or a hyperideal tetrahedron if and only if $(\alpha_1,\dots,\alpha_6)$ is of the hyperideal type. 

By Lemma \ref{converge}, $U_r=U-\frac{4\pi\sqrt{-1}}{r}\log\big(\frac{r}{2}\big)+O(\frac{1}{r}),$ where  $U$ is defined by
\begin{equation}\label{term}
\begin{split}
U(\alpha_1,\dots,\alpha_6,\xi)=&\pi^2+\frac{1}{2}\sum_{i=1}^4\sum_{j=1}^3(\eta_j-\tau_i)^2-\frac{1}{2}\sum_{i=1}^4(\tau_i-\pi)^2\\
&+(\xi-\pi)^2-\sum_{i=1}^4(\xi-\tau_i)^2-\sum_{j=1}^3(\eta_j-\xi)^2\\
&-2\mathit{Li}_2(1)-\frac{1}{2}\sum_{i=1}^4\sum_{j=1}^3\mathit{Li}_2\big(e^{2i(\eta_j-\tau_i)}\big)+\frac{1}{2}\sum_{i=1}^4\mathit{Li}_2\big(e^{2i(\tau_i-\pi)}\big)\\
&-\mathit{Li}_2\big(e^{2i(\xi-\pi)}\big)+\sum_{i=1}^4\mathit{Li}_2\big(e^{2i(\xi-\tau_i)}\big)+\sum_{j=1}^3\mathit{Li}_2\big(e^{2i(\eta_j-\xi)}\big)\\
\end{split}
\end{equation}
on the region $\mathrm{B_{H,\mathbb C}}$ consisting of $(\alpha_1,\dots,\alpha_6,\xi)\in\mathbb C^7$ such that $(\mathrm{Re}(\alpha_1),\dots,\mathrm{Re}(\alpha_6))$ is of the hyperideal type and $\max\{\mathrm{Re}(\tau_i)\}\leqslant \mathrm{Re}(\xi)\leqslant \min\{\mathrm{Re}(\eta_j), 2\pi\}.$ 

Let 
$$\mathrm{B_H}=\mathrm{B_{H,\mathbb C}}\cap \mathbb R^7.$$
Then by (\ref{dilogLob}), for $(\alpha_1,\dots,\alpha_6,\xi)\in\mathrm{B_H},$
\begin{equation}\label{UV} 
U(\alpha_1,\dots,\alpha_6,\xi)=2\pi^2+2\sqrt{-1}\cdot V(\alpha_1,\dots,\alpha_6,\xi)
\end{equation}
for $V:\mathrm{B_H}\to\mathbb R$ defined by 
\begin{equation}\label{V}
\begin{split}
V(\alpha_1,\dots,\alpha_6,\xi)=\,&\delta(\alpha_1,\alpha_2,\alpha_3)+\delta(\alpha_1,\alpha_5,\alpha_6)+\delta(\alpha_2,\alpha_4,\alpha_6)+\delta(\alpha_3,\alpha_4,\alpha_5)\\
&-\Lambda(\xi)+\sum_{i=1}^4\Lambda(\xi-\tau_i)+\sum_{j=1}^3\Lambda(\eta_j-\xi)
\end{split}
\end{equation}
with $\delta$ defined by
$$\delta(x,y,z)=-\frac{1}{2}\Lambda\Big(\frac{x+y-z}{2}\Big)-\frac{1}{2}\Lambda\Big(\frac{y+z-x}{2}\Big)-\frac{1}{2}\Lambda\Big(\frac{z+x-y}{2}\Big)+\frac{1}{2}\Lambda\Big(\frac{x+y+z}{2}\Big).$$
As a consequence, we have
\begin{equation}\label{pU}
\begin{split}
\frac{\partial U}{\partial \xi}=2\sqrt{-1}\cdot\frac{\partial V}{\partial \xi}=2\sqrt{-1}\cdot\log\frac{\sin(-\xi)\sin(\eta_1-\xi)\sin(\eta_2-\xi)\sin(\eta_3-\xi)}{\sin(\xi-\tau_1)\sin(\xi-\tau_2)\sin(\xi-\tau_3)\sin(\xi-\tau_4)},
\end{split}
\end{equation}
and
\begin{equation}\label{pUa}
\begin{split}
\frac{\partial U}{\partial \alpha_1}=2\sqrt{-1}\cdot\frac{\partial V}{\partial \alpha_1}=&\frac{\sqrt{-1}}{2}\cdot\log\frac{\sin(\frac{\alpha_1+\alpha_2-\alpha_3}{2})\sin(\frac{\alpha_1+\alpha_3-\alpha_2}{2})\sin(\frac{\alpha_1+\alpha_5-\alpha_6}{2})\sin(\frac{\alpha_1+\alpha_6-\alpha_5}{2})}{\sin(\frac{\alpha_2+\alpha_3-\alpha_1}{2})\sin(\frac{\alpha_1+\alpha_2+\alpha_3}{2})\sin(\frac{\alpha_5+\alpha_6-\alpha_1}{2})\sin(\frac{\alpha_1+\alpha_5+\alpha_6}{2})}\\
&+\sqrt{-1}\cdot\log\frac{\sin(\xi-\tau_1)\sin(\xi-\tau_2)}{\sin(\eta_1-\xi)\sin(\eta_2-\xi)}.
\end{split}
\end{equation}

A result of Costantino\,\cite{C} shows that for each $\alpha=(\alpha_1,\dots,\alpha_6)$ of the hyperideal type, there exists a unique $\xi_{\alpha}$ so that $(\alpha,\xi_\alpha)\in\mathrm{B_H}$ and 
$$\frac{\partial V(\alpha,\xi)}{\partial \xi}\Big|_{\xi=\xi_\alpha}=0.$$
Indeed, oen can prove that for each $\alpha,$ $V$ is strictly concave down in $\xi$ with derivatives $\pm\infty$ at the boundary points of the interval of $\xi,$ hence there is a unique critical point which is the absolute maximum. Moreover, he shows that by the Murakami-Yano formula\,\cite{MY,U},
\begin{equation}\label{VV}
V(\alpha,\xi_\alpha)=\mathrm{Vol}(\Delta_{|\pi-\alpha|}),
\end{equation}
the volume of the hyperideal tetrahedron with dihedral angles $|\pi-\alpha_1,|,\dots, |\pi-\alpha_6|.$

As a consequence of (\ref{UV}), we have on $\mathrm{B_H}$
\begin{equation}\label{Costantino}
\frac{\partial U(\alpha,\xi)}{\partial \xi}\Big|_{\xi=\xi_\alpha}=0.
\end{equation}

Let $u_i=e^{\sqrt{-1}\alpha_i}$ for $i=1,\dots,6$ and let  $z=e^{-2\sqrt{-1}\xi}.$ Then a direct computation shows that
\begin{equation}\label{pUC}
\frac{\partial U}{\partial \xi}=2\sqrt{-1}\cdot\log\frac{(1-z)(1-zu_1u_2u_4u_5)(1-zu_1u_3u_4u_6)(1-zu_2u_3u_5u_6)}{(1-zu_1u_2u_3)(1-zu_1u_5u_6)(1-zu_2u_4u_6)(1-zu_3u_4u_5)}\quad\quad(\mathrm{mod}\ 4\pi),
\end{equation}
and 
\begin{equation}\label{pUaC}
\begin{split}
\frac{\partial U}{\partial \alpha_1}=&\frac{\sqrt{-1}}{2}\cdot\log\frac{(1-u_1u_2u_3^{-1})(1-u_1u_2^{-1}u_3)(1-u_1u_5u_6^{-1})(1-u_1u_5^{-1}u_6)}{u_1^4(1-u_1^{-1}u_2u_3)(1-u_1^{-1}u_2^{-1}u_3^{-1})(1-u_1^{-1}u_5u_6)(1-u_1^{-1}u_5^{-1}u_6^{-1})}\\
&+\sqrt{-1}\cdot\log\frac{u_4(1-zu_1u_2u_3)(1-zu_1u_5u_6)}{(1-zu_1u_2u_4u_5)(1-zu_1u_3u_4u_6)}\quad\quad\quad\quad\quad\quad\quad\quad\quad\quad(\mathrm{mod}\ \pi).
\end{split}
\end{equation}

For a fixed $\alpha$ so that $\mathrm{Re}(\alpha)$ is of the hyperideal type, consider the function $U_\alpha$ of $\xi$ defined by $U_\alpha(\xi)=U(\alpha,\xi).$ From (\ref{pUC}), if $\xi$ is a critical point of $U_\alpha,$ then as a necessary condition
$$\frac{(1-z)(1-zu_1u_2u_4u_5)(1-zu_1u_3u_4u_6)(1-zu_2u_3u_5u_6)}{(1-zu_1u_2u_3)(1-zu_1u_5u_6)(1-zu_2u_4u_6)(1-zu_3u_4u_5)}=1,$$
which is equivalent to the following quadratic equation 
\begin{equation}\label{qe}
Az^2+Bz+C=0,
\end{equation}
where
\begin{equation*}
\begin{split}
A=&u_1u_4+u_2u_5+u_3u_6-u_1u_2u_6-u_1u_3u_5-u_2u_3u_4-u_4u_5u_6+u_1u_2u_3u_4u_5u_6,\\
B=&-\Big(u_1-\frac{1}{u_1}\Big)\Big(u_4-\frac{1}{u_4}\Big)-\Big(u_2-\frac{1}{u_2}\Big)\Big(u_5-\frac{1}{u_5}\Big)-\Big(u_3-\frac{1}{u_3}\Big)\Big(u_6-\frac{1}{u_6}\Big),\\
C=&\frac{1}{u_1u_4}+\frac{1}{u_2u_5}+\frac{1}{u_3u_6}-\frac{1}{u_1u_2u_6}-\frac{1}{u_1u_3u_5}-\frac{1}{u_2u_3u_4}-\frac{1}{u_4u_5u_6}+\frac{1}{u_1u_2u_3u_4u_5u_6}.
\end{split}
\end{equation*}
Let 
\begin{equation}\label{za}
z_\alpha=e^{-2\sqrt{-1}\xi(\alpha)}=\frac{-B+\sqrt{B^2-4AC}}{2A}
\end{equation}
and
$$z^*_\alpha=e^{-2\sqrt{-1}\xi^*(\alpha)}=\frac{-B-\sqrt{B^2-4AC}}{2A}.$$
Since the region of $\alpha$ is simply connected, we can choose the branch of $\sqrt{B^2-4AC}$ by analytic continuation. 
Then we have 
$$\frac{\partial U_\alpha}{\partial \xi}\Big|_{\xi=\xi(\alpha)}=4k\pi \quad\text{and}\quad\frac{\partial U_\alpha}{\partial \xi}\Big|_{\xi=\xi^*(\alpha)}=4k^*\pi $$
for some integers $k$ and $k^*.$ 

A direct computation shows that if $\alpha_1=\alpha_2=\cdots=\alpha_6=\pi,$ then $\xi(\alpha)=\frac{7\pi}{4}\in\big(\frac{3\pi}{2},2\pi\big)=(\max\{\tau_i\},\min\{\eta_j,2\pi\})$ and $\xi^*(\alpha)=\frac{5\pi}{4}\notin(\frac{3\pi}{2},2\pi).$ Hence for $\alpha$ real and of the hyperideal type, $\xi(\alpha)$ coincides with $\xi_\alpha$ in Costantino's result, and by (\ref{Costantino}) and the continuity of  $\frac{\partial U}{\partial \xi},$ we have
$$\frac{\partial U_\alpha}{\partial \xi}\Big|_{\xi=\xi(\alpha)}=0$$
for any $\alpha$ so that $\mathrm{Re}(\alpha)$ is of the hyperideal type, ie, $\xi(\alpha)$ is a critical point of $U_\alpha.$ 

\begin{remark} 
At this point, we do not know whether for any $\alpha$ so that $\mathrm{Re}(\alpha)$ is of the hyperideal type, $(\alpha,\xi(\alpha))$ always lies in $\mathrm{B_{H,\mathbb C}}.$ In Section \ref{Asy}, we will show that it does when $\theta_1,\dots,\theta_6$ are sufficiently small.
\end{remark}

For $\alpha\in\mathbb C^6$ so that $(\alpha,\xi(\alpha))\in\mathrm{B_{H,\mathbb C}},$ we define
\begin{equation}\label{W}
W(\alpha)=U(\alpha,\xi(\alpha)).
\end{equation}
 Then we have the key result of this section.

\begin{theorem}\label{co-vol} For a partition $(I,J)$ of $\{1,\dots,6\}$ and a fixed $(\theta_j)_{j\in J},$ 
$$W\big((\pi\pm \sqrt{-1} l_i)_{i\in I},(\pi\pm \theta_j)_{j\in J}\big)=2\pi^2+2\sqrt{-1}\cdot\mathrm{Cov}\big((l_i)_{i\in I}\big)$$
for all $(l_i)_{i\in I}\in \mathbb R^I_{>0},$ where $\mathrm{Cov}$ is the co-volume function defined in Definition \ref{cov}.
\end{theorem}

\begin{proof} The proof follows the argument in \cite{MY, U, KM}: the key step is proving that $W$ satisfies the same differential identities as those of $\textrm{Cov}$ in Lemma \ref{Sch}. Without loss of generality assume that $1\in I.$ Let $W^*(\alpha)=U(\alpha,\xi^*(\alpha))-4k^*\pi \xi^*(\alpha).$ Then we have
\begin{equation}\label{former}
\frac{d W}{d \alpha_1}=\frac{\partial U}{\partial \alpha_1}\Big|_{\xi=\xi(\alpha)}+\frac{\partial U}{\partial \xi}\Big|_{\xi=\xi(\alpha)}\cdot\frac{\partial \xi(\alpha)}{\partial \alpha_1}=\frac{\partial U}{\partial \alpha_1}\Big|_{\xi=\xi(\alpha)},\end{equation}
and 
\begin{equation}\label{latter}
\frac{d W^*}{d \alpha_1}=\frac{\partial U}{\partial \alpha_1}\Big|_{\xi=\xi^*(\alpha)}+\frac{\partial U-4k^*\pi i\xi}{\partial \xi}\Big|_{\xi=\xi^*(\alpha)}\cdot\frac{\partial \xi(\alpha)}{\partial \alpha_1}=\frac{\partial U}{\partial \alpha_1}\Big|_{\xi=\xi^*(\alpha)}.
\end{equation}
Let 
$$F(\alpha)=\frac{1}{2}\big(W(\alpha)-W^*(\alpha)\big).$$
Then by (\ref{pUaC}), (\ref{former}) and (\ref{latter}),
\begin{equation*} 
\begin{split}
\frac{d F}{d \alpha_1}=&\frac{\sqrt{-1}}{2}\cdot\log\frac{(1-z_\alpha u_1u_2u_3)(1-z_\alpha u_1u_5u_6)(1-z^*_\alpha u_1u_2u_4u_5)(1-z^*_\alpha u_1u_3u_4u_6)}{(1-z^*_\alpha u_1u_2u_3)(1-z^*_\alpha u_1u_5u_6)(1-z_\alpha u_1u_2u_4u_5)(1-z_\alpha u_1u_3u_4u_6)}\quad\quad(\text{mod}\ \pi)
.\\
\end{split}
\end{equation*}

Let $\mathrm R$ and $\mathrm S$ respectively be the terms in $(1-z_\alpha u_1u_2u_3)(1-z_\alpha u_1u_5u_6)(1-z^*_\alpha u_1u_2u_4u_5)(1-z^*_\alpha u_1u_3u_4u_6)$ not containing and containing $\sqrt{B^2-4AC}.$ Then by a direct computation (see also \cite{MY, U}),
$$\mathrm S =\mathrm Q\big(u_1^{-1}-u_1\big)\sqrt{B^2-4AC}, $$
 where 
 $$\mathrm Q= \frac{1}{4A^2}u_1^2u_4^{-1}(u_4u_5-u_3)(u_3u_4-u_5)(u_2u_4-u_6)(u_4u_6-u_2),$$
and
 $$\mathrm R=8\mathrm QG_{34},$$
 where $G_{ij}$ is the $ij$-th cofactor of the Gram matrix $G$ in Section \ref{tetra}.
We also have
 $$B^2-4AC=16\det G.$$

If $\alpha_1=\pi+\sqrt{-1} l_1,$ then  $$\mathrm S=\mathrm Q\big(-2\sinh l_1\cdot\sqrt{16\det G}\big)=-8\mathrm Q\big(\sqrt{G_{34}^2-G_{33}G_{44}}\big).$$
Therefore,  we have
\begin{equation*}
\frac{d F}{d \alpha_1}=\frac{\sqrt{-1}}{2}\cdot\log\frac{G_{34}-\sqrt{G_{34}^2-G_{33}c_{44}}}{G_{34}+\sqrt{G_{34}^2-G_{33}G_{44}}} \quad\quad(\text{mod } \pi).
\end{equation*}
Since
$$\cos \theta_1 =\frac{G_{34}}{\sqrt{G_{33}G_{44}}},$$
we have
$$\frac{G_{34}-\sqrt{G_{34}^2-G_{33}c_{44}}}{G_{34}+\sqrt{G_{34}^2-G_{33}G_{44}}}=e^{-2\sqrt{-1}\theta_1}.$$
Then
\begin{equation*}
\frac{d F}{d \alpha_1}=\frac{\sqrt{-1}}{2}\cdot\log e^{-2\sqrt{-1}\theta_1}=\theta_1 \quad\quad(\text{mod } \pi)
\end{equation*}
and
\begin{equation}\label{partialViii}
\frac{\partial F}{\partial l_1}=\sqrt{-1}\cdot\frac{d F}{d \alpha_1}=\sqrt{-1}\theta_1 \quad\quad(\text{mod } \sqrt{-1}\pi ).
\end{equation}

 If $\alpha_1=\pi-\sqrt{-1} l_1,$ then 
  $$\mathrm S=\mathrm Q\big(2\sinh l_1\cdot\sqrt{16\det G}\big)=8\mathrm Q\big(\sqrt{G_{34}^2-G_{33}G_{44}}\big),$$
and
\begin{equation}\label{partialViiii}
\frac{\partial F}{\partial l_1}=-\sqrt{-1}\cdot\frac{d F}{d \alpha_1}=\frac{1}{2}\log\frac{G_{34}+\sqrt{G_{34}^2-G_{33}c_{44}}}{G_{34}-\sqrt{G_{34}^2-G_{33}G_{44}}}=\frac{1}{2}\log e^{2\sqrt{-1}\theta_1}=\sqrt{-1}\theta_1 \quad\quad(\text{mod } \sqrt{-1}\pi).
\end{equation}

Below a direct computation (see also \cite{MY}) shows that
\begin{equation}\label{partialW}
\frac{\partial}{\partial l_1}\big(W(\alpha)+W^*(\alpha)\big)=0 \quad\quad(\text{mod } \sqrt{-1}\pi).
\end{equation}
Indeed, it comes from the following calculation:
$$(1-z_\alpha u_1u_2u_4u_5)(1-z^*_\alpha u_1u_2u_4u_5)=\frac{1}{A}\frac{(u_1u_2u_4u_5)^2}{u_3u_6}\Big(1-\frac{u_3}{u_4u_5}\Big)\Big(1-\frac{u_6}{u_2u_4}\Big)\Big(1-\frac{u_6}{u_1u_5}\Big)\Big(1-\frac{u_3}{u_1u_2}\Big),$$
$$(1-z_\alpha u_1u_3u_4u_6)(1-z^*_\alpha u_1u_3u_4u_6)=\frac{1}{A}\frac{(u_1u_3u_4u_6)^2}{u_2u_5}\Big(1-\frac{u_2}{u_4u_6}\Big)\Big(1-\frac{u_5}{u_3u_4}\Big)\Big(1-\frac{u_2}{u_1u_3}\Big)\Big(1-\frac{u_5}{u_1u_6}\Big),$$
$$(1-z_\alpha u_1u_2u _3)(1-z^*_\alpha u_1u_2u_3)=\frac{1}{A }\frac{(u_1u_2u_3)^2}{u_4u_5u_6}\Big(1-\frac{u_4u_5}{u_3}\Big)\Big(1-\frac{u_4u_6}{u_2}\Big)\Big(1-\frac{u_5u_6}{u_1}\Big)\Big(1-\frac{1}{u_1u_2u_3}\Big),$$ and
$$(1-z_\alpha u_1u_5u_6)(1-z^*_\alpha u_1u_5u_6)=\frac{1}{A }\frac{(u_1u_5u_6)^2}{u_2u_3u_4}\Big(1-\frac{u_2u_4}{u_6}\Big)\Big(1-\frac{u_3u_4}{u_5}\Big)\Big(1-\frac{u_2u_3}{u_1}\Big)\Big(1-\frac{1}{u_1u_5u_6}\Big).$$
These, together with (\ref{pUaC}), (\ref{former}) and (\ref{latter}), imply (\ref{partialW}).

Putting  (\ref{partialViii}), (\ref{partialViiii}) and (\ref{partialW}) together, we have 
$$\frac{\partial W}{\partial l_1}=\sqrt{-1}\theta_1+\sqrt{-1} k\pi$$
for some integer $k.$
To find the value of $k,$ we take $\alpha_i=\pi$ for all $i\in I,$ ie, $l_i=0$ for all $i\in I.$ Then by (\ref{pUa}), $\frac{\partial W}{\partial l_1}=\sqrt{-1}\cdot\frac{\partial W}{\partial \alpha_1}$ is real. Also, by (\ref{cos1}), if $l_1=0$ then $\theta_1=0.$ This implies that $k=0,$ and 
\begin{equation}
\frac{\partial W}{\partial l_1}=\sqrt{-1}\theta_1.
\end{equation}

By exactly the same argument, we have 
\begin{equation}\label{Sch2}
\frac{\partial W}{\partial l_i}=\sqrt{-1}\theta_i
\end{equation}
for all $i\in I.$ 

By (\ref{UV}), (\ref{VV}) and Definition \ref{cov}, 
$$W\big((\pi)_{i\in I},(\pi\pm\theta_j)_{j\in J}\big)=2\pi^2+2\sqrt{-1}\cdot\mathrm{Vol}(\Delta_{(\mathbf 0;\theta_J)})=2\pi^2+2\sqrt{-1}\cdot\mathrm{Cov}(\mathbf 0).$$
Together with Lemma \ref{Sch} and (\ref{Sch2}), we have the result.
\end{proof}

\begin{remark} Theorem \ref{co-vol} also plays an essential role in the proof of the main theorems of \cite{WY2, Y}. 
\end{remark}

\begin{corollary}\label{cv1} Let $\Delta$ be the deeply truncated tetrahedron of type $(I,J)$ with $\{l_i\}_{i\in I}$ the lengths of the edges of deep truncation and $\{\theta_j\}_{j\in J}$ the dihedral angles at the regular edges. Then
$$\mathrm{Vol}(\Delta)=\frac{1}{2}\mathrm{Im}\bigg(W-\sum_{i\in I} l_i\frac{\partial W}{\partial l_i}\bigg)\bigg|_{\big((\pi\pm \sqrt{-1} l_i)_{i\in I},(\pi\pm \theta_j)_{j\in J}\big)}.$$
\end{corollary}

Corollary \ref{cv1} is an immediate consequence of Theorem \ref{co-vol} and Lemma \ref{Sch}. See also \cite{KM} for a different volume formula involving both roots of the quadratic equation (\ref{qe}).

\section{Computation of the discrete Fourier transforms}

\begin{proposition}\label{computation}  Let $(I,J)$ be a partition of $\{1,\dots,6\}.$ Then the discrete Fourier transform $\mathrm{\widehat {Y}}_r(b_I; a_J)$ at the root of unity $q=e^{\frac{2\pi \sqrt{-1}}{r}}$ can be computed as 
$$\mathrm{\widehat {Y}}_r(b_I; a_J)=\frac{(-1)^{|I|\big(\frac{r}{2}+1\big)}\cdot n(a_J)}{4\{1\}^{|I|-2}}\sum_{a_I,k_1,k_2}\Big(\sum_{\epsilon_I}g_r^{\epsilon_I}(a_I,k_1,k_2)\Big),$$
where $n(a_J)$ is the number of $3$-admissible colorings $c$ such that $c_j \equiv a_j\ (\text{mod } 2)$ for each $j\in J,$ $\epsilon_I=(\epsilon_i)_{i\in I}\in\{1,-1\}^I$ runs over all multi-signs, $a_I=(a_i)_{i\in I}$ runs over all multi-even integers in $\{0,2\dots,r-3\}$ so that the triples $(a_1, a_2, a_3),$ $(a_1, a_5, a_6),$ $(a_2,a_4,a_6)$ and $(a_3,a_4,a_5)$ are $r$-admissible, and $k_1$ and $k_2$ run over all the integers in between $\max\{T_i\}$ and $\min\{Q_j,r-2\}$ with
$$g_r^{\epsilon_I}(a_I,k_1,k_2)=e^{\sum_{i\in I}\epsilon_i\cdot\frac{2\sqrt{-1}\pi(a_i+b_i+1)}{r} +\frac{r}{4\pi \sqrt{-1}}\mathcal W_r^{\epsilon_I}\big(\frac{2\pi a_I}{r},\frac{2\pi k_1}{r},\frac{2\pi k_2}{r}\big)}$$
where $\frac{2\pi a_I}{r}=\big(\frac{2\pi a_i}{r}\big)_{i\in I}$ and 
\begin{equation*}
\begin{split}
\mathcal W_r^{\epsilon_{I}}(\alpha_I,\xi_1,\xi_2)=&-\sum_{i\in I}2\epsilon_i(\alpha_i-\pi)(\beta_i-\pi)+U_r(\alpha_1,\alpha_2,\dots,\alpha_6,\xi_1)+U_r(\alpha_1,\alpha_2,\dots,\alpha_6,\xi_2)
\end{split}
\end{equation*}
where $\alpha_I=(\alpha_i)_{i\in I}.$
\end{proposition}

\begin{proof}  First, we observe that if we let the summation in the definition of $\mathrm{\widehat {Y}}_r(b_I; a_J)$ be over all multi-even integers $a_I$ instead of multi-integers, then the resulting quantity differs from $\mathrm{\widehat {Y}}_r(b_I; a_J)$ by a factor $n(a_J)$ by \cite[Lemma A.4, Theorem 2.9 and its proof]{DKY}.  Next, we observe that
$$(-1)^{a+b}q^{(a+1)(b+1)}=-(-1)^{\frac{r}{2}}q^{\big((a-\frac{r}{2})(b-\frac{r}{2})+a+b+1\big)},$$
and
$$(-1)^{a+b}q^{-(a+1)(b+1)}=(-1)^{\frac{r}{2}}q^{-\big((a-\frac{r}{2})(b-\frac{r}{2})+a+b+1\big)}.$$
As a consequence, we have
\begin{equation*}
\begin{split}
\mathrm H(a_i,b_i)=&\frac{1}{q-q^{-1}}\bigg((-1)^{a_i+b_i}q^{(a_i+1)(b_i+1)}-(-1)^{a_i+b_i}q^{-(a_i+1)(b_i+1)}\bigg)\\
=&\frac{-(-1)^{\frac{r}{2}}}{q-q^{-1}}\bigg(q^{\big((a_i-\frac{r}{2})(b_i-\frac{r}{2})+a_i+b_i+1\big)}+q^{-\big((a_i-\frac{r}{2})(b_i-\frac{r}{2})+a_i+b_i+1\big)}\bigg)\\
=&\frac{(-1)^{\frac{r}{2}+1}}{\{1\}}\sum_{\epsilon_i\in\{-1,1\}}q^{\epsilon_i\big((a_i-\frac{r}{2})(b_i-\frac{r}{2})+a_i+b_i+1\big)}\\
=&\frac{(-1)^{\frac{r}{2}+1}}{\{1\}}\sum_{\epsilon_i\in\{-1,1\}}e^{\epsilon_i\frac{2\sqrt{-1}\pi(a_i+b_i+1)}{r}+\frac{r}{4\pi\sqrt{-1}}\Big(-2\epsilon_i\big(\frac{2\pi a_i}{r}-\pi\big)\big(\frac{2\pi b_i}{r}-\pi\big)\Big)},\\
\end{split}
\end{equation*}
and hence
\begin{equation*}
\begin{split}\prod_{i\in I}\mathrm H(a_i,b_i)=&\frac{(-1)^{|I|\big(\frac{r}{2}+1\big)}}{\{1\}^{|I|}}\sum_{\epsilon_I\in\{-1,1\}^{|I|}}e^{\sum_{i\in I}\epsilon_i\frac{2\sqrt{-1}\pi(a_i+b_i+1)}{r}+\frac{r}{4\pi\sqrt{-1}}\sum_{i\in I}\Big(-2\epsilon_i\big(\frac{2\pi a_i}{r}-\pi\big)\big(\frac{2\pi b_i}{r}-\pi\big)\Big)}\\
=&\frac{(-1)^{|I|\big(\frac{r}{2}+1\big)}}{\{1\}^{|I|}}\sum_{\epsilon_I\in\{-1,1\}^{|I|}}e^{\sum_{i\in I}\epsilon_i\sqrt{-1}(\alpha_i+\beta_i+\frac{2\pi}{r})+\frac{r}{4\pi\sqrt{-1}}\sum_{i\in I}\big(-2\epsilon_i(\alpha_i-\pi)(\beta_i-\pi)\big)}.\\
\end{split}
\end{equation*}
Then the result follows from Proposition \ref{6jqd}.
\end{proof}

We notice that the summation in Proposition \ref{computation} is finite, and to use the Poisson Summation Formula, we need an infinite sum over integral points. To this end, we consider the following regions and a bump function over them. 

Let $\alpha_i=\frac{2\pi a_i}{r}$ for $i=1,\dots,6,$ $\beta_i=\frac{2\pi b_i}{r}$ for $i\in I,$ $\xi_s=\frac{2\pi k_s}{r}$ for $s=1,2,$ $\tau_i=\frac{2\pi T_i}{r}$ for $i=1,\dots,4,$ and $\eta_j=\frac{2\pi Q_j}{r}$ for $j=1,2,3.$ For a fixed $(\alpha_j)_{j\in J},$ let
$$\mathrm {D_A}=\Big\{(\alpha_I,\xi_1,\xi_2)\in\mathbb R^{|I|+2}\ \Big|\ (\alpha_1,\alpha_2,\dots,\alpha_6) \text{ is admissible, } \max\{\tau_i\}\leqslant \xi_s\leqslant \min\{\eta_j, 2\pi\}, s=1,2\Big\},$$
and let
$$\mathrm {D_H}=\Big\{(\alpha_I,\xi_1,\xi_2)\in\mathrm {D_A} \ \Big|\ (\alpha_1,\alpha_2,\dots,\alpha_6) \text{ is of the hyperideal type} \Big\}.$$
For a sufficiently small $\delta >0,$ let 
$$\mathrm {D_H^\delta}=\Big\{(\alpha_I,\xi_1,\xi_2)\in\mathrm {D_H}\ \Big|\ d((\alpha_I,\xi_1,\xi_2), \partial\mathrm {D_H})>\delta \Big\},$$
where $d$ is the Euclidean distance on $\mathbb R^n.$
We let $\psi:\mathbb R^{|I|+2}\to\mathbb R$ be a $C^{\infty}$-smooth bump function supported on $(\mathrm{D_H}, \mathrm{D_H^\delta}),$ ie,  \begin{equation*}
\left \{\begin{array}{rl}
\psi(\alpha_I,\xi_1,\xi_2)=1, & (\alpha_I,\xi_1,\xi_2)\in \overline{\mathrm{D_H^\delta}}\\
0<\psi(\alpha_I,\xi_1,\xi_2)<1, &  (\alpha_I,\xi_1,\xi_2)\in \mathrm{D_H}\setminus \overline{\mathrm{D_H^\delta}}\\
\psi(\alpha_I,\xi_1,\xi_2)=0, & (\alpha_I,\xi_1,\xi_2)\notin \mathrm{D_H},\\
\end{array}\right.
\end{equation*}
and let 
$$f^{\epsilon_I}_r(a_I,k_1,k_2)=\psi\Big(\frac{2\pi a_I}{r},\frac{2\pi k_1}{r},\frac{2\pi k_2}{r}\Big)g^{\epsilon_I}_r(a_I,k_1,k_2).$$

In Proposition \ref{computation}, the sum is over multi-even integers  $a_I.$ On the other hand, to use the Poisson Summation Formula, we need a sum over all integers. For this purpose, we for each $i\in I$ let $a_i=2a_i',$  $a_I'=(a_i')_{i\in I}$ and denote $(2a_i')_{i\in I}$ by $2a_I'.$ Then by Proposition \ref{computation}, 
$$\mathrm{\widehat {Y}}_r(b_I;a_J)=\frac{(-1)^{|I|\big(\frac{r}{2}+1\big)}\cdot n(a_J)}{4\{1\}^{|I|-2}}\sum_{(a_I',k_1,k_2)\in\mathbb Z^{|I|+2}}\Big(\sum_{\epsilon_I\in\{1,-1\}^I} f_r^{\epsilon_I}\big(2a_I',k_1,k_2\big)\Big)+\text{error term}.$$
Let $$f_r=\sum_{\epsilon_I\in\{1,-1\}^I} f_r^{\epsilon_I}.$$ Then
$$\mathrm{\widehat {Y}}_r(b_I;a_J)=\frac{(-1)^{|I|\big(\frac{r}{2}+1\big)}\cdot n(a_J)}{4\{1\}^{|I|-2}}\sum_{(a_I',k_1,k_2)\in\mathbb Z^{|I|+2}}f_r\big(2a_I',k_1,k_2\big)+\text{error term}.$$

Since $f_r$ is $C^{\infty}$-smooth and equals zero out of $\mathrm{D_H},$ it is in the Schwartz space on $\mathbb R^{|I|+2}.$ Then by the Poisson Summation Formula (see e.g. \cite[Theorem 3.1]{SS}),
$$\sum_{(a_I',k_1,k_2)\in\mathbb Z^{|I|+2}}f_r\big(2a_I',k_1,k_2\big)=\sum_{(m_I,n_1,n_2)\in\mathbb Z^{|I|+2}}\widehat {f_r}(m_I,n_1,n_2),$$
where $m_I=(m_i)_{i\in I}\in \mathbb Z^I$ and  $\widehat f_r(m_I,n_1,n_2)$ is the $(m_I,n_1,n_2)$-th Fourier coefficient of $f_r$ defined by
\begin{equation*}
\begin{split}
\widehat {f_r}(m_I,n_1,n_2)=\int_{\mathbb R^{|I|+2}}&f_r\big(2a_I',k_1,k_2\big)e^{\sum_{i\in I}2\pi \sqrt{-1}m_ia_i'+2\pi \sqrt{-1}n_1k_1+2\pi \sqrt{-1}n_2k_2}da'_Idk_1dk_2,
\end{split}
\end{equation*}
where $da'_I=\prod_{i\in I}da'_i.$

By the change of variable, and by changing $2a_i'$ back to $a_i,$ the Fourier coefficients can be computed as
\begin{proposition}\label{4.2}
$$\widehat{f_r}(m_I,n_1,n_2)=\sum_{\epsilon_I\in\{1,-1\}^I} \widehat{f^{\epsilon_I}_r}(m_I,n_1,n_2)$$
with
\begin{equation*}
\begin{split}
\widehat{f^{\epsilon_I}_r}(m_I,n_1,n_2)=\frac{r^{|I|+2}}{2^{2|I|+2}\cdot\pi^{|I|+2}}&\int_{\mathrm{D_H}}\psi(\alpha_I,\xi_1,\xi_2)e^{\sum_{i\in I}\epsilon_i\sqrt{-1}(\alpha_i+\beta_i+\frac{2\pi}{r})}\\ 
&\cdot e^{\frac{r}{4\pi \sqrt{-1}}\big(\mathcal W_r^{\epsilon_I}(\alpha_I,\xi_1,\xi_2)-\sum_{i\in I}2\pi m_i\alpha_i-4\pi n_1\xi_1-4\pi n_2\xi_2\big)}d\alpha_Id\xi_1d\xi_2,
\end{split}
\end{equation*}
where $d\alpha_I=\prod_{i\in I}d\alpha_i$ and
$$\mathcal W_r^{\epsilon_I}(\alpha_I,\xi_1,\xi_2)=-\sum_{i\in I}2\epsilon_i(\alpha_i-\pi)(\beta_i-\pi)+U_r(\alpha_1,\dots,\alpha_6,\xi_1)+U_r(\alpha_1,\dots,\alpha_6,\xi_2).$$
In particular,
\begin{equation*}
\begin{split}
\widehat{f^{\epsilon_I}_r}(0,\dots,0)=\frac{r^{|I|+2}}{2^{2|I|+2}\cdot\pi^{|I|+2}}\int_{\mathrm{D_H}}\psi(\alpha_I,\xi_1,\xi_2)e^{\sum_{i\in I}\epsilon_i\sqrt{-1}(\alpha_i+\beta_i+\frac{2\pi}{r})+\frac{r}{4\pi \sqrt{-1}}\mathcal W_r^{\epsilon_I}(\alpha_I,\xi_1,\xi_2)}d\alpha_Id\xi_1d\xi_2.
\end{split}
\end{equation*}
\end{proposition}

\begin{proposition}\label{Poisson}
$$\mathrm{\widehat {Y}}_r(b_I; a_J)=\frac{(-1)^{|I|\big(\frac{r}{2}+1\big)}\cdot n(a_J)}{4\{1\}^{|I|-2}}\sum_{(m_I,n_1,n_2)\in\mathbb Z^{|I|+2}}\widehat{ f_r}(m_I,n_1,n_2)+\text{error term}.$$
\end{proposition}

We will estimate the leading Fourier coefficients, the non-leading Fourier coefficients and the error term respectively in Sections \ref{leading}, \ref{ot} and \ref{ee}, and prove Theorem \ref{main} in Section \ref{pf}.


\section{Asymptotics}\label{Asy}

\begin{proposition}\label{saddle}
Let $D_{\mathbf z}$ be a region in $\mathbb C^n$ and let $D_{\mathbf a}$ be a region in $\mathbb R^k.$ Let $f(\mathbf z,\mathbf a)$ and $g(\mathbf z,\mathbf a)$ be complex valued functions on $D_{\mathbf z}\times D_{\mathbf a}$  which are holomorphic in $\mathbf z$ and smooth in $\mathbf a.$ For each positive integer $r,$ let $f_r(\mathbf z,\mathbf a)$ be a complex valued function on $D_{\mathbf z}\times D_{\mathbf a}$ holomorphic in $\mathbf z$ and smooth in $\mathbf a.$
For a fixed $\mathbf a\in D_{\mathbf a},$ let $f^{\mathbf a},$ $g^{\mathbf a}$ and $f_r^{\mathbf a}$ be the holomorphic functions  on $D_{\mathbf z}$ defined by
$f^{\mathbf a}(\mathbf z)=f(\mathbf z,\mathbf a),$ $g^{\mathbf a}(\mathbf z)=g(\mathbf z,\mathbf a)$ and $f_r^{\mathbf a}(\mathbf z)=f_r(\mathbf z,\mathbf a).$ Suppose $\{\mathbf a_r\}$ is a convergent sequence in $D_{\mathbf a}$ with $\lim_r\mathbf a_r=\mathbf a_0,$ $f_r^{\mathbf a_r}$ is of the form
$$ f_r^{\mathbf a_r}(\mathbf z) = f^{\mathbf a_r}(\mathbf z) + \frac{\upsilon_r(\mathbf z,\mathbf a_r)}{r^2},$$
$\{S_r\}$ is a sequence of embedded real $n$-dimensional closed disks in $D_{\mathbf z}$ sharing the same boundary, and $\mathbf c_r$ is a point on $S_r$ such that $\{\mathbf c_r\}$ is convergent in $D_{\mathbf z}$ with $\lim_r\mathbf c_r=\mathbf c_0.$ If for each $r$
\begin{enumerate}[(1)]
\item $\mathbf c_r$ is a critical point of $f^{\mathbf a_r}$ in $D_{\mathbf z},$
\item $\mathrm{Re}f^{\mathbf a_r}(\mathbf c_r) > \mathrm{Re}f^{\mathbf a_r}(\mathbf z)$ for all $\mathbf z \in S\setminus \{\mathbf c_r\},$
\item the Hessian matrix $\mathrm{Hess}(f^{\mathbf a_r})$ of $f^{\mathbf a_r}$ at $\mathbf c_r$ is non-singular,
\item $|g^{\mathbf a_r}(\mathbf c_r)|$ is bounded from below by a positive constant independent of $r,$
\item $|\upsilon_r(\mathbf z, \mathbf a_r)|$ is bounded from above by a constant independent of $r$ on $D_{\mathbf z},$ and
\item  the Hessian matrix $\mathrm{Hess}(f^{\mathbf a_0})$ of $f^{\mathbf a_0}$ at $\mathbf c_0$ is non-singular,
\end{enumerate}
then
\begin{equation*}
\begin{split}
 \int_{S_r} g^{\mathbf a_r}(\mathbf z) e^{rf_r^{\mathbf a_r}(\mathbf z)} d\mathbf z= \Big(\frac{2\pi}{r}\Big)^{\frac{n}{2}}\frac{g^{\mathbf a_r}(\mathbf c_r)}{\sqrt{-\det\mathrm{Hess}(f^{\mathbf a_r})(\mathbf c_r)}} e^{rf^{\mathbf a_r}(\mathbf c_r)} \Big( 1 + O \Big( \frac{1}{r} \Big) \Big).
 \end{split}
 \end{equation*}
\end{proposition}

A proof can be found in \cite[Appendix]{WY2}.

For fixed $\{\beta_i\}_{i\in I}$ and $\{\alpha_j\}_{j\in J},$ let $\theta_i=|\pi-\beta_i|$ for $i\in I$ and let $\theta_j=|\pi-\alpha_j|$ for $j\in J.$ The function $\mathcal W_r^{\epsilon_I}$ is approximated by the following function
$$\mathcal W^{\epsilon_I}(\alpha_I,\xi_1,\xi_2)=-\sum_{i\in I}2\epsilon_i(\alpha_i-\pi)(\beta_i-\pi)+U(\alpha_1,\dots,\alpha_6,\xi_1))+U(\alpha_1,\dots,\alpha_6,\xi_2).$$
The approximation will be specified in the proof of Proposition \ref{critical}. Notice that $\mathcal W^{\epsilon_I}$ is continuous on 
$$\mathrm{D_{H,\mathbb C}}=\big\{(\alpha_I,\xi_1,\xi_2)\in\mathbb C^{|I|+2}\ \big|\ (\mathrm{Re}(\alpha_I),\mathrm{Re}(\xi_1),\mathrm{Re}(\xi_2))\in \mathrm{D_{H}}\big\}$$ and for any $\delta>0$ is analytic on 
$$\mathrm{D^\delta_{H,\mathbb C}}=\big\{(\alpha_I,\xi_1,\xi_2)\in\mathbb C^{|I|+2}\ \big|\ (\mathrm{Re}(\alpha_I),\mathrm{Re}(\xi_1),\mathrm{Re}(\xi_2))\in \mathrm{D^\delta_{H}}\big\},$$ 
where $\mathrm{Re}(\alpha_I)=(\mathrm{Re}(\alpha_i))_{i\in I}.$

In the rest of this paper, we assume that 	$\theta_1,\dots,\theta_6$ are sufficiently close to $0,$ or equivalently, $\{\beta_i\}_{i\in I}$ and $\{\alpha_j\}_{j\in J}$ are sufficiently close to $\pi.$ In the special case $\beta_i=\alpha_j=\pi$ for all $i\in I$ and $j\in J,$ a direct computation shows that $\xi(\pi,\dots,\pi)=\frac{7\pi}{4}.$ We denote by $\pi_I$ the point $(\pi, \dots, \pi)\in \mathbb C^I.$ For $\delta>0,$ we denote by $\mathrm{D_{\delta,\mathbb C}}$  the $L^1$ $\delta$-neighborhood  of $\big(\pi_I,\frac{7\pi}{4},\frac{7\pi}{4}\big)$ in $\mathbb C^{|I|+2},$  that is 
$$\mathrm{D_{\delta,\mathbb C}}=\Big\{(\alpha_I,\xi_1,\xi_2)\in \mathbb C^{|I|+2}\ \Big|\ d_{L^1}\Big((\alpha_I,\xi_1,\xi_2),\Big(\pi_I,\frac{7\pi}{4},\frac{7\pi}{4}\Big)\Big)<\delta\Big\},$$
where $d_{L^1}$ is the real $L^1$ norm on $\mathbb C^n$ defined by
$$d_{L^1}(\mathbf x,\mathbf y)=\max_{i\in\{1,\dots,n\}}\{|\mathrm {Re}(x_i)-\mathrm{Re}(y_i)|, |\mathrm {Im}(x_i)-\mathrm{Im}(y_i)| \},$$
where $\mathbf x=(x_1,\dots,x_n)$ and $\mathbf y=(y_1,\dots,y_n).$  We will also consider the region 
$$\mathrm{D_{\delta}}=\mathrm{D_{\delta,\mathbb C}}\cap \mathbb R^{|I|+2}.$$

\subsection{Critical points and critical values of \texorpdfstring{$\mathcal W^{\epsilon_I}$}{W}}

Let $\Delta(\theta_I; \theta_J)$ be the deeply truncated tetrahedron of type $(I,J)$ with edges of deep truncation $\{e_i\}_{i\in I}$ and regular edges $\{e_j\}_{j\in J},$  and with $\theta_i$ the dihedral angle at $e_i$ for $i\in I,$  and  with $\theta_i$  the dihedral angle at $e_j$ for $j\in J.$ 
For $i\in I,$ let  $l_i$ be the length of $e_i.$

\begin{proposition}\label{crit} Suppose $\{\beta_i\}_{i\in I}$ and $\{\alpha_j\}_{j\in J}$ are sufficiently close to $\pi.$ Let $\theta_i=|\pi-\beta_i|$ for $i\in I$ and let $\theta_j=|\pi-\alpha_j|$ for $j\in J.$  Let $\mu_i=1$ if $\beta_i\geqslant\pi$ and let $\mu_i=-1$ if $\beta_i\leqslant \pi$ for $i\in I.$ Then $\mathcal W^{\epsilon_I}(\alpha_I,\xi_1,\xi_2)$ has a critical point 
$$z^{\epsilon_I}=\Big(\big(\pi+\epsilon_i\mu_i\sqrt{-1}l_i\big)_{i\in I}, \xi\big((\pi+\epsilon_i\mu_i\sqrt{-1}l_i)_{i\in I},\alpha_J\big), \xi\big((\pi+\epsilon_i\mu_i\sqrt{-1}l_i)_{i\in I},\alpha_J\big)\Big)$$
in $\mathrm{D_{\delta,\mathbb C}}$ with critical value $$4\pi^2+4\sqrt{-1}\cdot \mathrm{Vol}\big(\Delta(\theta_I;\theta_J)\big).$$
\end{proposition}

\begin{proof}  
By (\ref{cos1}), for $i\in I,$ $l_i$ is sufficiently small for a sufficiently small $\theta_i.$  Then by the continuity of $\xi(\alpha_1,\dots,\alpha_6),$ the point $((\pi\pm il_i)_{i\in I},\xi((\pi\pm il_i)_{i\in I},\alpha_J),\xi((\pi\pm il_i)_{i\in I},\alpha_J))\in \mathrm D_{\delta,\mathbb C}$ for sufficiently small $\theta_1,\dots,\theta_6.$

For any $\alpha_I=(\alpha_i)_{i\in I}$ so that $(\alpha_I,\xi(\alpha_I,\alpha_J),\xi(\alpha_I,\alpha_J))\in\mathrm{D_{H,\mathrm C}},$ and for $s=1,2,$
\begin{equation}\label{partialxi}
\frac{\partial \mathcal W^{\epsilon_I}}{\partial \xi_s}\Big|_{(\alpha_I,\xi(\alpha_I,\alpha_J),\xi(\alpha_I,\alpha_J))}=\frac{\partial U}{\partial \xi}\Big|_{((\alpha_I,\alpha_J),\xi(\alpha_I,\alpha_J))}=\frac{\partial U_{(\alpha_I, \alpha_J)}}{\partial \xi}\Big|_{\xi(\alpha_I, \alpha_J)}=0.
\end{equation}
In particular, 
$$\frac{\partial \mathcal W^{\epsilon_I}}{\partial \xi_1}\Big|_{z^{\epsilon_I}}=\frac{\partial \mathcal W^{\epsilon_I}}{\partial \xi_2}\Big|_{z^{\epsilon_I}}=0.$$

Let $\alpha=(\alpha_I,\alpha_J)$ and let $W(\alpha)=U(\alpha,\xi(\alpha))$ be the function defined in (\ref{W}). Then for $i\in I$
$$\frac{\partial  W}{\partial \alpha_i}\Big|_\alpha=\frac{\partial U}{\partial \alpha_i}\Big|_{(\alpha,\xi(\alpha))}+\frac{\partial U}{\partial \xi}\Big|_{(\alpha,\xi(\alpha))}\cdot\frac{\partial \xi(\alpha)}{\partial \alpha_i}\Big|_\alpha=\frac{\partial U}{\partial \alpha_i}\Big|_{(\alpha,\xi(\alpha))}.$$
Together with Theorem \ref{co-vol} and Lemma \ref{Sch}, we have 
\begin{equation*}
\begin{split}
\frac{\partial U}{\partial \alpha_i}\Big|_{\big(\big((\pi+\epsilon_i\mu_i\sqrt{-1}l_i)_{i\in I},\alpha_J\big),\xi\big((\pi+\epsilon_i\mu_i \sqrt{-1}l_i)_{i\in I},\alpha_J\big)\big)}=&\frac{\partial  W}{\partial \alpha_i}\Big|_{\big((\pi+\epsilon_i\mu_i \sqrt{-1}l_i)_{i\in I},\alpha_J\big)}\\
=&-\epsilon_i\mu_i \sqrt{-1}\cdot\frac{\partial W}{\partial l_i}\Big|_{(l_i)_{i\in I}}=\epsilon_i\mu_i \theta_i.
\end{split}
\end{equation*}
Then we have for $i\in I,$ 
\begin{equation*}
\begin{split}
\frac{\partial \mathcal W^{\epsilon_I}}{\partial \alpha_i}\Big|_{z^{\epsilon_I}}=&-2\epsilon_i(\beta_i-\pi)+2\frac{\partial U}{\partial \alpha_i}\Big|_{\big(\big((\pi+\epsilon_i\mu_i\sqrt{-1}l_i)_{i\in I},\alpha_J\big),\xi\big((\pi+\epsilon_i\mu_i \sqrt{-1}l_i)_{i\in I},\alpha_J\big)\big)}\\
=&2\epsilon_i\big(-(\beta_i-\pi)+\mu_i\theta_i\big)=0,
\end{split}
\end{equation*}
where the last equality comes from that $\mu_i\theta_i=\beta_i-\pi.$

For the critical value, by Theorem \ref{co-vol} and $\mu_i\theta_i=\beta_i-\pi$ again, we have 
\begin{equation*}
\begin{split}
\mathcal W^{\epsilon_I}(z^{\epsilon_I})=&-\sum_{i\in I}2\epsilon_i(\sqrt{-1}\epsilon_i\mu_il_i)(\beta_i-\pi)+2\Big(2\pi^2+2\sqrt{-1}\Big(\mathrm{Vol}(\Delta(\theta_I;\theta_J))+\frac{1}{2}\sum_{i\in I}\theta_il_i\Big)\Big)\\
=&4\pi^2+4\sqrt{-1}\cdot \mathrm{Vol}((\Delta(\theta_I;\theta_J))+\sum_{i\in I}2\sqrt{-1}\big(-\mu_i(\beta_i-\pi)+\theta_i\big)l_i\\
=&4\pi^2+4\sqrt{-1}\cdot \mathrm{Vol}((\Delta(\theta_I;\theta_J)).
\end{split}
\end{equation*}
\end{proof}


\subsection{Convexity of \texorpdfstring{$\mathcal W^{\epsilon_I}$}{W}}

\begin{proposition}\label{convexity} 
There exists a $\delta_0>0$ such that if all $\{\alpha_j\}_{j\in J}$ are in $(\pi-\delta_0,\pi+\delta_0),$ then for any $\epsilon_I,$ $\mathrm{Im}\mathcal W^{\epsilon_I}(\alpha_I,\xi_1,\xi_2)$ is strictly concave down in $\{\mathrm{Re}(\alpha_i)\}_{i\in I},$ $\mathrm{Re}(\xi_1)$ and $\mathrm{Re}(\xi_2),$ and is strictly concave up in $\{\mathrm{Im}(\alpha_i)\}_{i\in I},$ $\mathrm{Im}(\xi_1)$ and $\mathrm{Im}(\xi_2)$ on $\mathrm{D_{\delta_0,\mathbb C}}.$
\end{proposition}

\begin{proof} We first consider the special case $\{\alpha_i\}_{i\in I},$ $\xi_1$ and $\xi_2$ are real.  In this case, 
$$\mathrm{Im}\mathcal W^{\epsilon_I}(\alpha_I,\xi_1,\xi_2)=2V(\alpha_1,\dots,\alpha_6,\xi_1)+2V(\alpha_1,\dots,\alpha_6,\xi_2)$$
with $V(\alpha_1,\dots,\alpha_6,\xi)$ defined in (\ref{V}).

At $\big(\pi,\dots,\pi,\frac{7\pi}{4},\frac{7\pi}{4}\big),$ we have $\frac{\partial ^2\mathrm{Im}\mathcal W^{\epsilon_I}}{\partial \alpha_i^2} =-4$ for $i\in I,$ $\frac{\partial ^2\mathrm{Im}\mathcal W^{\epsilon_I}}{\partial \alpha_i\alpha_{i'}} =-2$ for $i\neq i'\in I,$ $\frac{\partial ^2\mathrm{Im}\mathcal W^{\epsilon_I}}{\partial \alpha_i\xi_s} =4$ for $i\in I$ and $s=1,2,$ $\frac{\partial ^2\mathrm{Im}\mathcal W^{\epsilon_I}}{\partial \xi_s^2} =-16$ for $s=1,2$ and $\frac{\partial ^2\mathrm{Im}\mathcal W^{\epsilon_I}}{\partial \xi_1\xi_2} =0.$ Then a direct computation shows that,  at $\big(\pi,\dots,\pi,\frac{7\pi}{4},\frac{7\pi}{4}\big),$ the Hessian matrix 
of $\mathrm{Im}\mathcal W^{\epsilon_I}$ is negative definite.

Then by the continuity, there exists a sufficiently small $\delta_0>0$ such that for all $\{\alpha_j\}_{j\in J}$ in $(\pi-\delta_0,\pi+\delta_0)$ and $(\alpha_I,\xi_1,\xi_2)\in \mathrm{D_{\delta_0,\mathbb C}},$ the Hessian matrix of $\mathrm{Im}\mathcal W^{\epsilon_I}$ with respect to $\{\mathrm{Re}(\alpha_i)\}_{i\in I},$ $\mathrm{Re}(\xi_1)$ and $\mathrm{Re}(\xi_2)$ is still negative definite, implying that $\mathrm{Im}\mathcal W^{\epsilon_I}$ is strictly concave down in $\{\mathrm{Re}(\alpha_i)\}_{i\in I},$ $\mathrm{Re}(\xi_1)$ and $\mathrm{Re}(\xi_2)$ on $\mathrm{D_{\delta_0,\mathbb C}}.$ Since $\mathcal W^{\epsilon_I}$ is holomorphic, $\mathrm{Im}\mathcal W^{\epsilon_I}$ is strictly concave up in $\{\mathrm{Im}(\alpha_i)\}_{i\in I},$ $\mathrm{Im}(\xi_1)$ and $\mathrm{Im}(\xi_2)$ on $\mathrm{D_{\delta_0,\mathbb C}}.$
\end{proof}

\begin{proposition}\label{nonsingular} If all $\{\alpha_j\}_{j\in J}$ are in $(\pi-\delta_0,\pi+\delta_0),$ then
the Hessian matrix $\mathrm{Hess}\mathcal W^{\epsilon_I}$ of $\mathcal W^{\epsilon_I}$ with respect to $\{\alpha_i\}_{i\in I},$ $\xi_1$ and $\xi_2$ is non-singular on $\mathrm{D_{\delta_0,\mathbb C}}.$
\end{proposition}

\begin{proof} By Proposition \ref{convexity}, the real part of the $\mathrm{Hess}\mathcal W^{\epsilon_I}$ is negative definite. Then by \cite[Lemma]{L}, it is nonsingular.
\end{proof}


\subsection{Asymptotics of the leading Fourier coefficients}\label{leading}

\begin{proposition}\label{critical}
Suppose $\{\beta_i\}_{i\in I}$ and $\{\alpha_j\}_{j\in J}$ are in $\{\pi-\epsilon,\pi+\epsilon
\}$ for a sufficiently small $\epsilon>0.$ For $\epsilon_I\in\{1,-1\}^I,$ let  $z^{\epsilon_I}$ be the critical point of $\mathcal W^{\epsilon_I}$ described in Proposition \ref{crit}. Then
$$\widehat{f^{\epsilon_I}_r}(0,\dots,0)=\frac{C^{\epsilon_I}(z^{\epsilon_I})}{\sqrt{-\det\mathrm{Hess}\Big(\frac{\mathcal W^{\epsilon_I}(z^{\epsilon_I})}{4\pi\sqrt{-1}}\Big)}}e^{\frac{r}{2\pi}2\mathrm{Vol}(\Delta(\theta_I;\theta_J))}\Big( 1 + O \Big( \frac{1}{r} \Big) \Big)$$
where each $C^{\epsilon_I}(z^{\epsilon_I})$ depends continuously on $\{\beta_i
\}_{i\in I}$ and $\{\alpha_j\}_{j\in J};$ and when $\beta_i=\alpha_j=\pi,$
$$C^{\epsilon_I}(z^{\epsilon_I})=\frac{r^{\frac{|I|}{2}-1}}{2^{\frac{3|I|}{2}+1}\pi^{\frac{|I|}{2}+1}}.$$
\end{proposition}

 For the proof of Proposition \ref{critical}, we need the following

\begin{lemma}\label{absm} For each $\epsilon_I\in\{1,-1\}^I$ and any fixed $\{\alpha_j\}_{j\in J},$ 
$$\max_{\mathrm{D_H}}\mathrm{Im}\mathcal W^{\epsilon_I} \leqslant \mathrm{Im}\mathcal W^{\epsilon_I}\Big(\pi,\dots,\pi,\frac{7\pi}{4},\frac{7\pi}{4}\Big)=4\mathrm{Vol}(\Delta_{(0,\dots,0)})$$
where $\Delta_{(0,\dots,0)}$ is the regular ideal octahedron, and the equality holds if and only if $\alpha_1=\dots=\alpha_6=\pi$ and $\xi_1=\xi_2=\frac{7\pi}{4}.$
\end{lemma}

\begin{proof} On $\mathrm{D_H},$ we have 
$$\mathrm{Im}\mathcal W^{\epsilon_I}(\alpha_I,\xi_1,\xi_2)=2V(\alpha_1,\dots,\alpha_6,\xi_1)+2V(\alpha_1,\dots,\alpha_6,\xi_2)$$ 
for $V$ defined in (\ref{V}). Then the result is a consequence of the result of Costantino\,\cite{C} and the Murakami-Yano formula\,\cite{MY} (see Ushijima\,\cite{U} for the case of hyperideal tetrahedra). Indeed, by \cite{C}, for a fixed $\alpha=(\alpha_1,\dots,\alpha_6)$ of the hyperideal type, the function $f(\xi)$ defined by $f(\xi)=V(\alpha,\xi)$ is strictly concave down and the unique maximum point $\xi(\alpha)$ exists and lies in $(\max\{\tau_i\},\min\{\eta_j,2\pi\}),$ ie, $(\alpha,\xi(\alpha))\in\mathrm{B_H}.$ Then by \cite{U}, $V(\alpha,\xi(\alpha))=\mathrm{Vol}(\Delta_{|\pi-\alpha|}),$ the volume of the hyperideal tetrahedron $\Delta_{|\pi-\alpha|}$ with dihedral angles $|\pi-\alpha_1|,\dots, |\pi-\alpha_6|.$ Since $\xi(\pi,\dots,\pi)=\frac{7\pi}{4}$ and $\Delta_{(0,\dots,0)}$ has the maximum volume among all the hyperideal tetrahedra, $V\big(\pi,\dots,\pi,\frac{7\pi}{4}\big)=\mathrm{Vol}(\Delta_{(0,\dots,0)})\geqslant \mathrm{Vol}(\Delta_{|\pi-\alpha|})=V(\alpha,\xi(\alpha))\geqslant V(\alpha,\xi)$ for any $(\alpha,\xi)\in \mathrm{B_H}.$ 

For the equality part, suppose $(\alpha_1,\dots,\alpha_6,\xi_1,\xi_2)\neq \big(\pi,\dots,\pi,\frac{7\pi}{4},\frac{7\pi}{4}\big).$ If $(\alpha_1,\dots,\alpha_6)\neq (\pi,\dots,\pi),$ then 
$\mathrm{Im}\mathcal W^{\epsilon_I}(\alpha_I,\xi_1,\xi_2)\leqslant 4\mathrm{Vol}(\Delta_{|\pi-\alpha|})<4\mathrm{Vol}(\Delta_{(0,\dots,0)}).$ If $(\alpha_1,\dots,\alpha_6)=(\pi,\dots,\pi)$ but, say, $\xi_1\neq \frac{7\pi}{4},$ then the strict concavity of $f(\xi)$ implies that $\mathrm{Im}\mathcal W^{\epsilon_I}(\pi,\dots,\pi,\xi_1,\xi_2)< \mathrm{Im}\mathcal W^{\epsilon_I}\big(\pi,\dots,\pi,\frac{7\pi}{4},\frac{7\pi}{4}\big).$
\end{proof}

\begin{proof}[Proof of Proposition \ref{critical}] Let $\delta_0>0$ be as in Proposition \ref{convexity}.
By Lemma \ref{convexity}, Proposition \ref{absm} and the compactness of $\mathrm{D_H}\setminus\mathrm{D_{\delta_0}},$
$$4\mathrm{Vol}(\Delta_{(0,\dots,0)})>\max_{\mathrm{D_H}\setminus\mathrm{D_{\delta_0}}} \mathrm{Im}\mathcal W^{\epsilon_I}.$$
By Proposition \ref{crit} and continuity, if $\{\beta_i\}_{i\in I}$ and $\{\alpha_j\}_{j\in J}$ are sufficiently close to $\pi,$ then the critical point $z^{\epsilon_I}$ of $\mathcal W^{\epsilon_I}$  as in Proposition \ref{crit} lies in $\mathrm{D_{\delta_0,\mathbb C}},$ and $\mathrm{Im}\mathcal W^{\epsilon_I}(z^{\epsilon_I})=4\mathrm{Vol}(\Delta(\theta_I;\theta_J))$ is sufficiently close to $4\mathrm{Vol}(\Delta_{(0,\dots,0)})$ so that
 $$\mathrm{Im}\mathcal W^{\epsilon_I}(z^{\epsilon_I})>\max_{\mathrm{D_H}\setminus\mathrm{D_{\delta_0}}} \mathrm{Im}\mathcal W^{\epsilon_I}.$$
 
Therefore, we only need to estimate  the integral on $\mathrm{D_{\delta_0}}.$ To do this, we consider as drawn in Figure \ref{surface} the surface $S^{\epsilon_I}=S^{\epsilon_I}_{\text{top}}\cup S^{\epsilon_I}_{\text{side}}$ in $\overline{\mathrm{D_{\delta_0,\mathbb C}}},$ where
$$S^{\epsilon_I}_{\text{top}}=\{ (\alpha_I,\xi_1,\xi_2)\in \mathrm{D_{\delta_0,\mathbb C}}\ |\ ((\mathrm{Im}(\alpha_I)),\mathrm{Im}(\xi_1),\mathrm{Im}(\xi_2))=\mathrm{Im}(z^{\epsilon_I})\}$$
and
$$S^{\epsilon_I}_{\text{side}}=\{ (\alpha_I,\xi_1,\xi_2)+t\sqrt{-1}\cdot \mathrm{Im}(z^{\epsilon_I})\ |\ (\alpha_I,\xi_1,\xi_2)\in\partial \mathrm{D_{\delta_0}},t\in[0,1]\}.$$

\begin{figure}[htbp]
\centering
\includegraphics[scale=0.3]{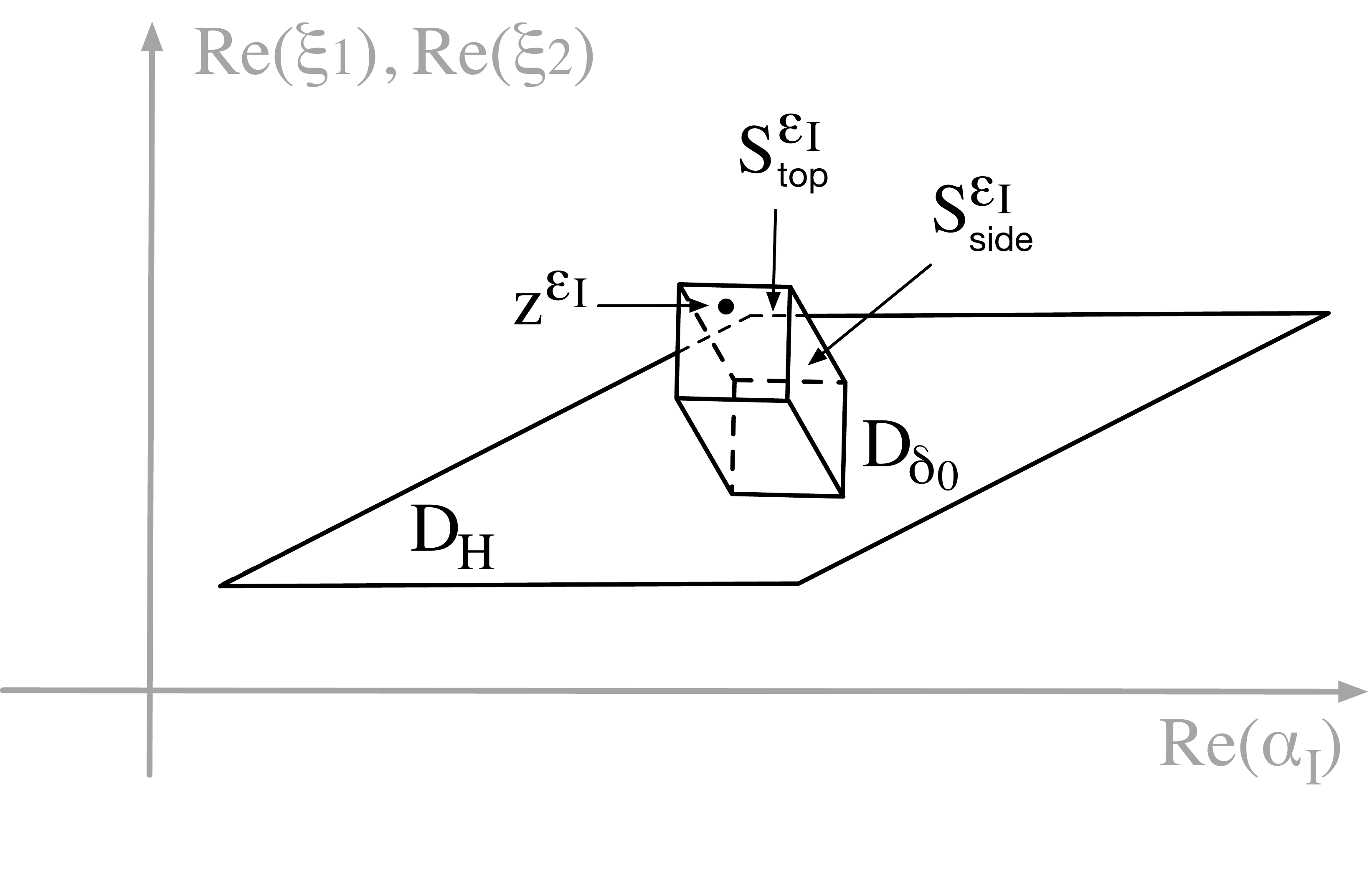}
\caption{The deformed surface $S^{\epsilon_I}$}
\label{surface} 
\end{figure}

By analyticity, the integral remains the same if we deform the domain from $\mathrm{D_{\delta_0}}$ to $S^{\epsilon_I}.$

By Proposition \ref{convexity}, $\mathrm{Im}\mathcal W^{\epsilon_I}$  is concave down on $S^{\epsilon_I}_{\text{top}}.$ Since $z^{\epsilon_I}$ is the critical points of $\mathrm{Im}\mathcal W^{\epsilon_I},$ it is the only absolute maximum on $S^{\epsilon_I}_{\text{top}}.$

On the side $S^{\epsilon_I}_{\text{side}},$ for each $(\alpha_I,\xi_1,\xi_2)\in \partial \mathrm{D_{\delta_0}},$ consider the function 
$$g^{\epsilon_I}_{(\alpha_I,\xi_1,\xi_2)}(t)= \mathrm{Im}\mathcal W^{\epsilon_I}((\alpha_I,\xi_1,\xi_2)+t\sqrt{-1}\cdot \mathrm{Im}(z^{\epsilon_I}))$$
on $[0,1].$ By Lemma \ref{convexity}, $g^{\epsilon_I}_{(\alpha_I,\xi_1,\xi_2)}(t)$ is concave up for any $(\alpha_I,\xi_1,\xi_2)\in \partial \mathrm{D_{\delta_0}}.$ As a consequence, $g^{\epsilon_I}_{(\alpha_I,\xi_1,\xi_2)}(t)\leqslant \max\{g^{\epsilon_I}_{(\alpha_I,\xi_1,\xi_2)}(0), g^{\epsilon_I}_{(\alpha_I,\xi_1,\xi_2)}(1)\}.$ Now by the previous two steps, since $(\alpha_I,\xi_1,\xi_2)\in \partial \mathrm{D_{\delta_0}},$ 
$$g^{\epsilon_I}_{(\alpha_I,\xi_1,\xi_2)}(0)= \mathrm{Im}\mathcal W^{\epsilon_I}(\alpha_I,\xi_1,\xi_2)<\mathrm{Im}\mathcal W^{\epsilon_I}(z^{\epsilon_I});$$
and since $(\alpha_I,\xi_1,\xi_2)+\sqrt{-1}\cdot \mathrm{Im}(z^{\epsilon_I})\in S^{\epsilon_I}_{\text{top}},$
$$g^{\epsilon_I}_{(\alpha_I,\xi_1,\xi_2)}(0)= \mathrm{Im}\mathcal W^{\epsilon_I}((\alpha_I,\xi_1,\xi_2)+\sqrt{-1}\cdot \mathrm{Im}(z^{\epsilon_I}))<\mathrm{Im}\mathcal W^{\epsilon_I}(z^{\epsilon_I}).$$ As a consequence,
 $$\mathrm{Im}\mathcal W^{\epsilon_I}(z^{\epsilon_I})>\max_{S^{\epsilon_I}_{\text{side}}} \mathrm{Im}\mathcal W^{\epsilon_I}.$$

Therefore, we proved that $z^{\epsilon_I}$ is the unique maximum point of $\mathrm{Im}\mathcal W^{\epsilon_I}$ on $S^{\epsilon_I}\cup\big( \mathrm{D_H}\setminus\mathrm{D_{\delta_0}}\big),$ and $\mathcal W^{\epsilon_I}$ has critical value $4\pi^2+4\sqrt{-1}\cdot\mathrm{Vol}(\Delta(\theta_I;\theta_J)$ at $z^{\epsilon_I}.$

By Proposition \ref{nonsingular}, $\det\mathrm{Hess}\mathcal W^{\epsilon_I}(z^{\epsilon_I})\neq 0.$

Finally,  we estimate the difference between $\mathcal W^{\epsilon_I}_r$ and $\mathcal W^{\epsilon_I}.$ By Lemma \ref{converge}, (3), we have
 $$\varphi_r\Big(\frac{\pi}{r}\Big)=\mathrm{Li}_2(1)+\frac{2\pi\sqrt{-1}}{r}\log\Big(\frac{r}{2}\Big)-\frac{\pi^2}{r}+O\Big(\frac{1}{r^2}\Big);$$
and for $z$ with $0<\mathrm{Re} z<\pi$ have
 $$\varphi_r\Big(z+\frac{k\pi}{r}\Big)=\varphi_r(z)+\varphi'_r(z)\cdot\frac{k\pi}{r}+O\Big(\frac{1}{r^2}\Big).$$
 Then by Lemma \ref{converge}, in $\big\{(\alpha_I,\xi_1,\xi_2)\in \overline{\mathrm{D_{H,\mathbb C}^\delta}}\ \big|\ |\mathrm{Im}(\alpha_i)| < L\text{ for } i\in I, |\mathrm{Im}(\xi_1)| < L, |\mathrm{Im}(\xi_2)| < L\}$ for some $L>0,$
 $$\mathcal W^{\epsilon_I}_r(\alpha_I,\xi_1,\xi_2)=\mathcal W^{\epsilon_I}(\alpha_I,\xi_1,\xi_2)-\frac{8\pi\sqrt{-1}}{r}\log\Big(\frac{r}{2}\Big)+\frac{4\pi \sqrt{-1} \cdot\kappa(\alpha_I,\xi_1,\xi_2)}{r}+\frac{\nu_r(\alpha_I,\xi_1,\xi_2)}{r^2},$$
with
 \begin{equation*}
 \begin{split}
&\kappa(\alpha_I,\xi_1,\xi_2)\\
=&\sum_{i=1}^4 \sqrt{-1}\tau_i- \sqrt{-1}\xi_1- \sqrt{-1}\xi_2-3\sqrt{-1}\pi \\
&+\frac{1}{2}\sum_{i=1}^4\sum_{j=1}^3\log\big(1-e^{2\sqrt{-1}(\eta_j-\tau_i)}\big)-\frac{3}{2}\sum_{i=1}^4\log\big(1-e^{2\sqrt{-1}(\tau_i-\pi)}\big)\\
&+\frac{3}{2}\log\big(1-e^{2\sqrt{-1}(\xi_1-\pi)}\big)-\frac{1}{2}\sum_{i=1}^4\log\big(1-e^{2\sqrt{-1}(\xi_1-\tau_i)}\big)-\frac{1}{2}\sum_{j=1}^3\log\big(1-e^{2\sqrt{-1}(\eta_j-\xi_1)}\big)\\
&+\frac{3}{2}\log\big(1-e^{2\sqrt{-1}(\xi_2-\pi)}\big)-\frac{1}{2}\sum_{i=1}^4\log\big(1-e^{2\sqrt{-1}(\xi_2-\tau_i)}\big)-\frac{1}{2}\sum_{j=1}^3\log\big(1-e^{2\sqrt{-1}(\eta_j-\xi_2)}\big)\\
 \end{split}
 \end{equation*}
 and  $|\nu_r(\alpha_I,\xi_1,\xi_2)|$ bounded from above by a constant independent of $r.$
 Then
  \begin{equation*}
 \begin{split}
&e^{\sum_{i\in I}\epsilon_i\sqrt{-1}\big(\alpha_i+\beta_i+\frac{2\pi}{r}\big)+\frac{r}{4\pi \sqrt{-1}}{\mathcal W}^{\epsilon_I}_r(\alpha_I,\xi_1,\xi_2)}\\
=&\Big(\frac{r}{2}\Big)^{-2}e^{\sum_{i\in I}\epsilon_i\sqrt{-1}(\alpha_i+\beta_i)+\kappa(\alpha_I,\xi_1,\xi_2)}\cdot e^{\frac{r}{4\pi \sqrt{-1}}\Big(\mathcal W^{\epsilon_I}(\alpha_I,\xi_1,\xi_2)+\frac{\nu_r(\alpha_I,\xi_1,\xi_2)-\sum_{i\in I}\epsilon_i 8\pi^2}{r^2}\Big)}.
 \end{split}
 \end{equation*}

Now let $D_{\mathbf z}=\big\{(\alpha_I,\xi_1,\xi_2)\in \overline{\mathrm{D_{H,\mathbb C}^\delta}}\ \big|\ |\mathrm{Im}(\alpha_i)| < L\text{ for } i\in I, |\mathrm{Im}(\xi_1)| < L, |\mathrm{Im}(\xi_2)| < L\}$ for some $L>0,$ Let $\mathbf a_r=((\beta_i)_{i\in I},(\alpha_j)_{j\in J})$ (recall that $\beta_i=\frac{2\pi n_i}{r}$ and $\alpha_j=\frac{2\pi m_j}{r}$ depends on $r$),
$f^{\mathbf a_r}(\alpha_I,\xi_1,\xi_2)=\frac{\mathcal W^{\epsilon_I}(\alpha_I,\xi_1,\xi_2)}{4\pi\sqrt{-1}},$ $g^{\mathbf a_r}(\alpha_I,\xi_1,\xi_2)=\psi(\alpha_I,\xi_1,\xi_2)e^{\sum_{i\in I}\epsilon_i\sqrt{-1}(\alpha_i+\beta_i)+\kappa(\alpha_I,\xi_1,\xi_2)},$ $f_r^{\mathbf a_r}(\alpha_I,\xi_1,\xi_2)=\frac{{\mathcal W}_r^{\epsilon_I}(\alpha_I,\xi_1,\xi_2)}{4\pi\sqrt{-1}}-\frac{2}{r}\log\big(\frac{r}{2}\big),$ $\upsilon_r(\alpha_I,\xi_1,\xi_2)=\nu_r(\alpha_I,\xi_1,\xi_2)-\sum_{i\in I}\epsilon_i 8\pi^2,$ $S_r=S^{\epsilon_I}\cup \big(\mathrm{D_H}\setminus\mathrm{D_{\delta_0}}\big)$ and $z^{\epsilon_I}$ is the critical point of $f$ in $D_{\mathbf z}.$  Then all the conditions of Proposition \ref{saddle} are satisfied and the result follows. 

When $\beta_i=\alpha_j=\pi,$ a direct computation shows that 
\begin{equation*}
\begin{split}
C^{\epsilon_I}(z^{\epsilon_I})=\frac{r^{|I|+2}}{2^{2|I|+2}\pi^{|I|+2}}\Big(\frac{2\pi}{r}\Big)^{
\frac{|I|+2}{2}}\Big(\frac{r}{2}\Big)^{-2} g\Big(\pi,\dots,\pi,\frac{7\pi}{4},\frac{7\pi}{4}\Big)=\frac{r^{\frac{|I|}{2}-1}}{2^{\frac{3|I|}{2}+1}\pi^{\frac{|I|}{2}+1}}.
\end{split}
\end{equation*}
\end{proof}

\begin{corollary}\label{5.8} If $\epsilon>0$ is sufficiently small and all $\{\beta_i\}_{i\in I}$ and $\{\alpha_j\}_{j\in J}$ are in $\{\pi-\epsilon,\pi+\epsilon
\},$ then 
$$\sum_{\epsilon_I\in\{1,-1\}^I}\frac{C^{\epsilon_I}(z^{\epsilon_I})}{\sqrt{-\det\mathrm{Hess}\Big(\frac{\mathcal W^{\epsilon_I}(z^{\epsilon_I})}{4\pi\sqrt{-1}}\Big)}}\neq 0.$$
\end{corollary}

\begin{proof} If $\beta_i=\alpha_j=\pi$ for all $i\in I$ and $j\in J,$ then all $z^{\epsilon_I}=\big(\pi,\dots,\pi,\frac{7\pi}{4},\frac{7\pi}{4}\big)$ and all $\mathcal W^{\epsilon_I}$ are the same functions. As a consequence, all the $C^{\epsilon_I}(z^{\epsilon_I})$'s and all Hessian determinants $\det\mathrm{Hess}\Big(\frac{\mathcal W^{\epsilon_I}(z^{\epsilon_I})}{4\pi\sqrt{-1}}\Big)$'s are the same at this point, implying that the sum is not equal to zero. Then by continuity, if $\epsilon$ is small enough, then the sum remains non-zero.
\end{proof}

\begin{remark} We suspect that all $C^{\epsilon_I}(z^{\epsilon_I})$'s and all $\det\mathrm{Hess}\Big(\frac{\mathcal W^{\epsilon_I}(z^{\epsilon_I})}{4\pi\sqrt{-1}}\Big)$'s are always the same for any given $\{\beta_i\}_{i\in I}$ and $\{\alpha_j\}_{j\in J}.$
\end{remark}

\subsection{Estimate of the other Fourier coefficients}\label{ot}

\begin{proposition}\label{other} Suppose $\{\beta_i\}_{i\in I}$ and $\{\alpha_j\}_{j\in J}$ are in $\{\pi-\epsilon,\pi+\epsilon
\}$ for a sufficiently small $\epsilon>0.$  If $(m_I,n_1,n_2)\neq(0,\dots,0),$ then
$$\Big|\widehat{f^{\epsilon_I}_r}(m_I,n_1,n_2)\Big|<O\Big(e^{\frac{r}{2\pi}\big(2\mathrm{Vol}(\Delta(\theta_I;\theta_J))-\epsilon'\big)}\Big)$$
for some $\epsilon'>0.$
\end{proposition}

\begin{proof} Recall that if $\beta_i=\alpha_j=\pi$ for all $i\in I$ and $j\in J,$ then the total derivative 
$$D\mathcal W^{\epsilon_I}\Big(\pi,\dots,\pi,\frac{7\pi}{4}, \frac{7\pi}{4}\Big)=(0,\dots,0).$$
Hence there exists a $\delta_1>0$ and an $\epsilon>0$ such that if $\{\beta_i\}_{i\in I}$ and $\{\alpha_j\}_{j\in J}$ are in $\{\pi-\epsilon,\pi+\epsilon
\},$ then for all $(\alpha_I,\xi_1,\xi_2)\in D_{\delta_1,\mathbb C}$ and for any unit vector $\mathbf u=((u_i)_{i\in I},w_1,w_2)\in \mathbb R^{|I|+2},$ the directional derivatives 
$$|D_{\mathbf u}\mathrm{Im}\mathcal W^{\epsilon_I}(\alpha_I,\xi_1,\xi_2)|=\Big|\sum_{i\in I}u_i\frac{\partial \mathrm{Im}\mathcal W^{\epsilon_I}}{\partial \mathrm{Im}(\alpha_i)}+w_1\frac{\partial \mathrm{Im}\mathcal W^{\epsilon_I}}{\partial \mathrm{Im}(\xi_1)}+w_2\frac{\partial \mathrm{Im}\mathcal W^{\epsilon_I}}{\partial \mathrm{Im}(\xi_2)}\Big|<\frac{2\pi-\epsilon''}{2\sqrt{2|I|+4}}$$
for some $\epsilon''>0.$

On $\mathrm{D_H},$ we have 
\begin{equation*}
\begin{split}
 &\mathrm{Im}\Big(\mathcal W^{\epsilon_I}(\alpha_I,\xi_1,\xi_2)-\sum_{i\in I}2\pi m_i\alpha_i-4\pi n_1\xi_1-4\pi n_2\xi_2\Big)=\mathrm{Im}\mathcal W^{\epsilon_I}(\alpha_I,\xi_1,\xi_2).
\end{split}
\end{equation*}
Then by Lemma \ref{convexity}, Proposition \ref{absm} and the compactness of $\mathrm{D_H}\setminus\mathrm{D_{\delta_1}},$
$$4\mathrm{Vol}(\Delta_{(0,\dots,0)})>\max_{\mathrm{D_H}\setminus\mathrm{D_{\delta_1}}} \mathrm{Im}\Big(\mathcal W^{\epsilon_I}(\alpha_I,\xi_1,\xi_2)-\sum_{i\in I}2\pi m_i\alpha_i-4\pi n_1\xi_1-4\pi n_2\xi_2\Big)+\epsilon'''$$
for some $\epsilon'''>0.$
By Proposition \ref{crit} and continuity, if $\{\beta_i\}_{i\in I}$ and $\{\alpha_j\}_{j\in J}$ are sufficiently close to $\pi,$ then the critical point $z^{\epsilon_I}$ of $\mathcal W^{\epsilon_I}$  as in Proposition \ref{crit} lies in $\mathrm{D_{\delta_1,\mathbb C}},$ and $\mathrm{Im}\mathcal W^{\epsilon_I}(z^{\epsilon_I})=4\mathrm{Vol}(\Delta(\theta_I;\theta_J))$ is sufficiently close to $4\mathrm{Vol}(\Delta_{(0,\dots,0)})$ so that
\begin{equation}\label{b}
\mathrm{Im}\mathcal W^{\epsilon_I}(z^{\epsilon_I})>\max_{\mathrm{D_H}\setminus\mathrm{D_{\delta_1}}} \mathrm{Im}\Big(\mathcal W^{\epsilon_I}(\alpha_I,\xi_1,\xi_2)-\sum_{i\in I}2\pi m_i\alpha_i-4\pi n_1\xi_1-4\pi n_2\xi_2\Big)+\epsilon'''.
\end{equation}
 
Therefore, we only need to estimate  the integral on $\mathrm{D_{\delta_1}}.$ 

If $(m_I,n_1,n_2)\neq (0,\dots,0),$ then there is at least one of $\{m_i\}_{i\in I},$ $n_1$ and $n_2$ that is nonzero. Without loss of generality, assume that $m_1\neq 0.$

If $m_1>0,$ then consider the surface $S^+=S^+_{\text{top}}\cup S^+_{\text{side}}$ in $\overline{\mathrm{D_{\delta_1,\mathbb C}}}$ where
$$S^+_{\text{top}}=\{ (\alpha_I,\xi_1,\xi_2)\in \mathrm{D_{\delta_1,\mathbb C}}\ |\ (\mathrm{Im}(\alpha_I),\mathrm{Im}(\xi_1),\mathrm{Im}(\xi_2))=(\delta_1,0,\dots,0)\}$$
and
$$S^+_{\text{side}}=\{ (\alpha_I,\xi_1,\xi_2)+(t\sqrt{-1}\delta_1,0,\dots,0)\ |\ (\alpha_I,\xi_1,\xi_2)\in\partial \mathrm{D_{\delta_1}}, t\in[0,1]\}.$$      
On the top, for any $(\alpha_I,\xi_1,\xi_2)\in S^+_{\text{top}},$ by the Mean Value Theorem, 
\begin{equation*}
\begin{split}
\big|\mathrm{Im}\mathcal W^{\epsilon_I}(z^{\epsilon_I})-\mathrm{Im}\mathcal W^{\epsilon_I}(\alpha_I,\xi_1,\xi_2)\big|
=&\big|D_{\mathbf u}\mathrm{Im}\mathcal W^{\epsilon_I}(z)\big|\cdot\big\|z^{\epsilon_I}-(\alpha_I,\xi_1,\xi_2)\big\|\\
<&\frac{2\pi-\epsilon''}{2\sqrt{2|I|+4}}\cdot2\sqrt{2|I|+4} \delta_1\\
=&2\pi\delta_1-\epsilon''\delta_1,
\end{split}
\end{equation*}
where $z$ is some point on the line segment connecting $z^{\epsilon_I}$ and $(\alpha_I,\xi_1,\xi_2),$ $\mathbf u=\frac{z^{\epsilon_I}-(\alpha_I,\xi_1,\xi_2)}{\|z^{\epsilon_I}-(\alpha_I,\xi_1,\xi_2)\|}$ and $2\sqrt{2|I|+4} \delta_1$ is the diameter of $\mathrm{D_{\delta_1,\mathbb C}}.$
Then
\begin{equation*}
\begin{split}
\mathrm{Im}\Big(\mathcal W^{\epsilon_I}(\alpha_I,\xi_1,\xi_2)-\sum_{i\in I}2\pi m_i\alpha_i-4\pi n_1\xi_1-4\pi n_2\xi_2\Big)=&\mathrm{Im}\mathcal W^{\epsilon_I}(\alpha_I,\xi_1,\xi_2)-2\pi m_1\delta_1\\
<&\mathrm{Im}\mathcal W^{\epsilon_I}(z^{\epsilon_I})+2\pi\delta_1-\epsilon''\delta_1-2\pi\delta_1\\
=&\mathrm{Im}\mathcal W^{\epsilon_I}(z^{\epsilon_I})-\epsilon'' \delta_1.
\end{split}
\end{equation*}

On the side, for any point $(\alpha_I,\xi_1,\xi_2)+(t\sqrt{-1}\delta_1,0,\dots,0)\in S^+_{\text{side}},$ by the Mean Value Theorem again, we have
$$\big|\mathrm{Im}\mathcal W^{\epsilon_I}\big((\alpha_I,\xi_1,\xi_2)+(t\sqrt{-1}\delta_1,0,\dots,0)\big)-\mathrm{Im}\mathcal W^{\epsilon_I}(\alpha_I,\xi_1,\xi_2)\big|<\frac{2\pi-\epsilon''}{2\sqrt{2|I|+4}} t\delta_1.$$
Then 
\begin{equation*}
\begin{split}
\mathrm{Im}\mathcal W^{\epsilon_I}\big((\alpha_I,\xi_1,\xi_2)+(t\sqrt{-1}\delta_1,0,\dots,0)\big)-2\pi m_1 t\delta_1<&\mathrm{Im}\mathcal W^{\epsilon_I}(\alpha_I,\xi_1,\xi_2)+\frac{2\pi-\epsilon''}{2\sqrt{2|I|+4}} t\delta_1-2\pi t\delta_1\\
<&\mathrm{Im}\mathcal W^{\epsilon_I}(\alpha_I,\xi_1,\xi_2)\\
<&\mathrm{Im}\mathcal W^{\epsilon_I}(z^{\epsilon_I})-\epsilon''',
\end{split}
\end{equation*}
where the last inequality comes from the fact that $(\alpha_I,\xi_1,\xi_2)\in \partial \mathrm{D_{\delta_1}}\subset \mathrm{D_H}\setminus\mathrm{D_{\delta_1}}$ and (\ref{b}).

Now let $\epsilon'=\min\{\epsilon''\delta_1,\epsilon'''\},$ then on $S^+\cup \big(\mathrm{D_H}\setminus\mathrm{D_{\delta_1}}\big),$ 
$$\mathrm{Im}\Big(\mathcal W^{\epsilon_I}(\alpha_I,\xi_1,\xi_2)-\sum_{i\in I}2\pi m_i\alpha_i-4\pi n_1\xi_1-4\pi n_2\xi_2\Big)<\mathrm{Im}\mathcal W^{\epsilon_I}(z^{\epsilon_I})-\epsilon',$$
and the result follows.

If $m_1<0,$ then we consider the surface $S^-=S^-_{\text{top}}\cup S^-_{\text{side}}$ in $\overline{\mathrm{D_{\delta_1,\mathbb C}}}$ where
$$S^-_{\text{top}}=\{ (\alpha_I,\xi_1,\xi_2)\in \mathrm{D_{\delta_1,\mathbb C}}\ |\ (\mathrm{Im}(\alpha_I),\mathrm{Im}(\xi_1),\mathrm{Im}(\xi_2))=(-\delta_1,0,\dots,0)\}$$
and
$$S^-_{\text{side}}=\{ (\alpha_I,\xi_1,\xi_2)-(t\sqrt{-1}\delta_1,0,\dots,0)\ |\ (\alpha_I,\xi_1,\xi_2)\in\partial \mathrm{D_{\delta_1}}, t\in[0,1]\}.$$      
Then the same estimate as in the previous case proves that on
$S^-\cup \big(\mathrm{D_H}\setminus\mathrm{D_{\delta_1}}\big),$ 
$$\mathrm{Im}\Big(\mathcal W^{\epsilon_I}(\alpha_I,\xi_1,\xi_2)-\sum_{i\in I}2\pi m_i\alpha_i-4\pi n_1\xi_1-4\pi n_2\xi_2\Big)<\mathrm{Im}\mathcal W^{\epsilon_I}(z^{\epsilon_I})-\epsilon',$$
from which the result follows.
\end{proof}


\subsection{Estimate of the error term}\label{ee}

The goal of this section is to estimate the error term in Proposition \ref{Poisson}.

\begin{proposition}\label{error} Suppose $\{\alpha_j\}_{j\in J}$ are in $\{\pi-\epsilon,\pi+\epsilon
\}$ for a sufficiently small $\epsilon>0.$ Then
the error term in Proposition \ref{Poisson} is less than $O\big(e^{\frac{r}{2\pi}(2\mathrm{Vol}(\Delta(\theta_I;\theta_J))-\epsilon')}\big)$
for some $\epsilon'>0.$
\end{proposition}

For the proof we need the following estimate, which first appeared in \cite[Proposition 8.2]{GL} for $q=e^{\frac{\pi \sqrt{-1}}{r}},$ and for the root $q=e^{\frac{2\pi \sqrt{-1}}{r}}$ in \cite[Proposition 4.1]{DK}.

\begin{lemma}\label{est}
 For any integer $0<n<r$ and at $q=e^{\frac{2\pi \sqrt{-1}}{r}},$
 $$ \log\left|\{n\}!\right|=-\frac{r}{2\pi}\Lambda\left(\frac{2n\pi}{r}\right)+O\left(\log(r)\right).$$
\end{lemma}

\begin{proof}[Proof of Proposition \ref{error} ]
For a fixed $\alpha_J=(\alpha_j)_{j\in J},$ let 
\begin{equation*}
M_{\alpha_J}=\max\big\{V(\alpha_1,\dots,\alpha_6,\xi_1)+V(\alpha_1,\dots,\alpha_6,\xi_2)\ \big|\ (\alpha_I,\xi_1,\xi_2)\in\partial \mathrm {D_H}\cup\big(\mathrm {D_A}\setminus \mathrm{D_H}\big)\big\}
\end{equation*}
where $V$ is as defined in (\ref{V}). Then by \cite[Sections 3 \& 4]{BDKY}, 
$$M_{\alpha_J}<2v_8=2\mathrm{Vol}(\Delta_{(0,\dots,0)});$$
and by continuity, if $\epsilon$ is sufficiently small and $\{\theta_1,\dots,\theta_6\}$ are less than $\epsilon,$ then
$$M_{\alpha_J}<2\mathrm{Vol}(\Delta{(\theta_I;\theta_J)}).$$

Now by Lemma \ref{est} and the continuity, for $\epsilon'=\frac{2\mathrm{Vol}(\Delta{(\theta_I;\theta_J)})-M_{\alpha_J}}{2},$ we can choose a sufficiently small $\delta>0$ so that if $\big(\frac{2\pi a_I}{r},\frac{2\pi k_1}{r},\frac{2\pi k_2}{r}\big)\notin \mathrm{D_H^\delta},$ then
$$\Big|g_r^{\epsilon_I}(a_I, k_1, k_2)\Big|<O\Big(e^{\frac{r}{2\pi}(M_{\alpha_J}+\epsilon')}\Big)=O\Big(e^{\frac{r}{2\pi}(2\mathrm{Vol}(\Delta{(\theta_I;\theta_J)})-\epsilon')}\Big).$$
Let $\psi$ be the bump function supported on $(\mathrm{D_H}, \mathrm{D_H^{\delta}}).$ Then the error term in Proposition \ref{Poisson} is less than $O\big(e^{\frac{r}{2\pi}(2\mathrm{Vol}(\Delta{(\theta_I;\theta_J)})-\epsilon')}\big).$
\end{proof}


\subsection{Proof of Theorem \ref{main}}\label{pf}

\begin{proof}[Proof of Theorem \ref{main}] Let $\epsilon>0$ be sufficiently small so that the conditions of Propositions \ref{critical}, \ref{other} and  \ref{error} and of Corollary \ref{5.8} are satisfied, and suppose $\{\beta_i\}_{i\in I}$ and $\{\alpha_j\}_{j\in J}$ are all in $(\pi-\epsilon, \pi+\epsilon).$

By Propositions \ref{4.2}, \ref{Poisson}, \ref{critical}, \ref{other} and  \ref{error}, 
\begin{equation*}
\begin{split}
\mathrm{\widehat {Y}}_r(b_I; a_J)=&\frac{(-1)^{|I|\big(\frac{r}{2}+1\big)}\cdot n(a_J)}{4\{1\}^{|I|-2}}\Big(\sum_{\epsilon_I\in\{1,-1\}^I}\widehat{ f_r^{\epsilon_I}}(0,\dots,0)\Big)\Big(1+O\big(e^{\frac{r}{2\pi}{(-\epsilon')}}\big)\Big)\\
=&\frac{(-1)^{|I|\big(\frac{r}{2}+1\big)}\cdot n(a_J)}{4\{1\}^{|I|-2}}\bigg( \sum_{\epsilon_I\in\{1,-1\}^I}\frac{C^{\epsilon_I}(z^{\epsilon_I})}{\sqrt{-\det\mathrm{Hess}\Big(\frac{\mathcal W^{\epsilon_I}(z^{\epsilon_I})}{4\pi\sqrt{-1}}\Big)}}\bigg)  e^{\frac{r}{2\pi}2\mathrm{Vol}(\Delta(\theta_I;\theta_J))}\Big( 1 + O \Big( \frac{1}{r} \Big) \Big);
\end{split}
\end{equation*}
and by Corollary \ref{5.8},
 $$\sum_{\epsilon_I\in\{1,-1\}^I}\frac{C^{\epsilon_I}(z^{\epsilon_I})}{\sqrt{-\det\mathrm{Hess}\Big(\frac{\mathcal W^{\epsilon_I}(z^{\epsilon_I})}{4\pi\sqrt{-1}}\Big)}}\neq 0,$$
which completes the proof.
\end{proof}


\section{Numerical evidence for Conjecture \ref{conj}}\label{appendix}

In this appendix we show numerical evidence supporting Conjecture \ref{conj}. We provide calculations for two deeply truncated tetrahedra of type $((1),(23456))$ and one deeply truncated tetrahedron of every other type in Figure \ref{deep}. All the calculations are performed with the Mathematica software.

We provide plots for:
\begin{enumerate}[(1)]
 \item Every tetrahedron with five angles equal to $0$ and a deeply truncated edge with angle between $0$ and $\frac{\pi}{2}$ at $r=2017.$
 \item Two tetrahedra with a single deeply truncated edge.
 \item Two tetrahedra with two deeply truncated edges (one per type).
 \item Three tetrahedra with three deeply truncated edges (one per type).
\end{enumerate}

Because of Proposition \ref{prop:duality} this actually accounts for all partitions $(I,J)$ up to relabeling.

In every case we show, if a tetrahedron has angle $\alpha$ at an edge, the sequence of colorings we choose for that edge is $\left\lfloor \frac{r}{4\pi}(\pi-\alpha)\right\rfloor.$

The fact that the angles and lengths we list correspond to hyperbolic tetrahedra can be checked directly from the Gram matrix using Proposition \ref{prop:criterion}.

\vspace{.5cm}
\begin{minipage}{.45\textwidth}
  \includegraphics[width=.95\textwidth]{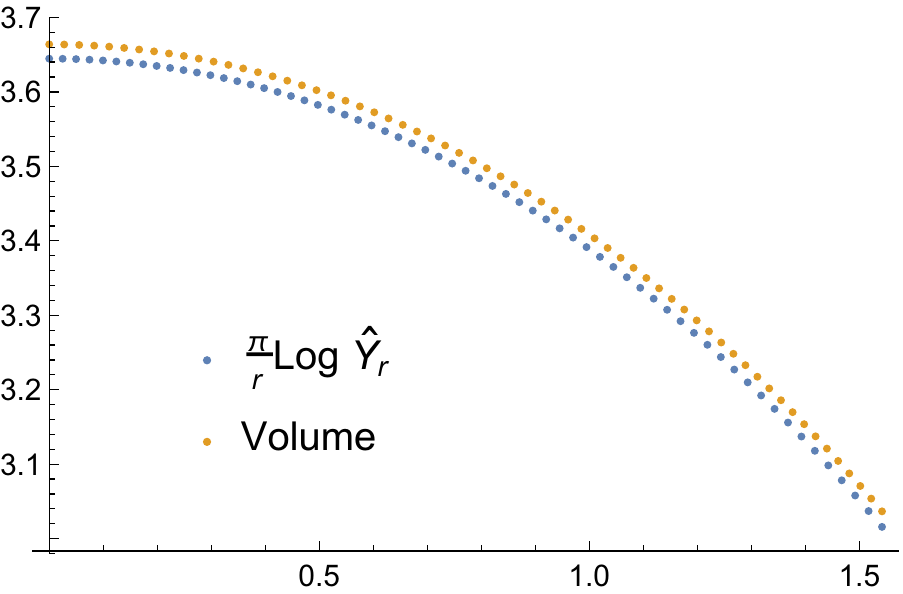}
\end{minipage}
\begin{minipage}{.45\textwidth}
 \begin{flushright}
 {\tabulinesep=1.2mm
 \begin{tabu}{|C{2.3cm}|C{4.5cm}|}
 \hline
  Type & $\left( (1),(23456)\right)$\\  \hline
  Angles & $(\alpha),\left(0,0,0,0,0\right)$ \\ \hline Edge lengths & $l$  \\ \hline Error & $<0.1\%$\\
  \hline
 \end{tabu}}

\end{flushright}
\end{minipage}

\vspace{1.5cm}

\begin{minipage}{.45\textwidth}
  \includegraphics[width=.95\textwidth,left]{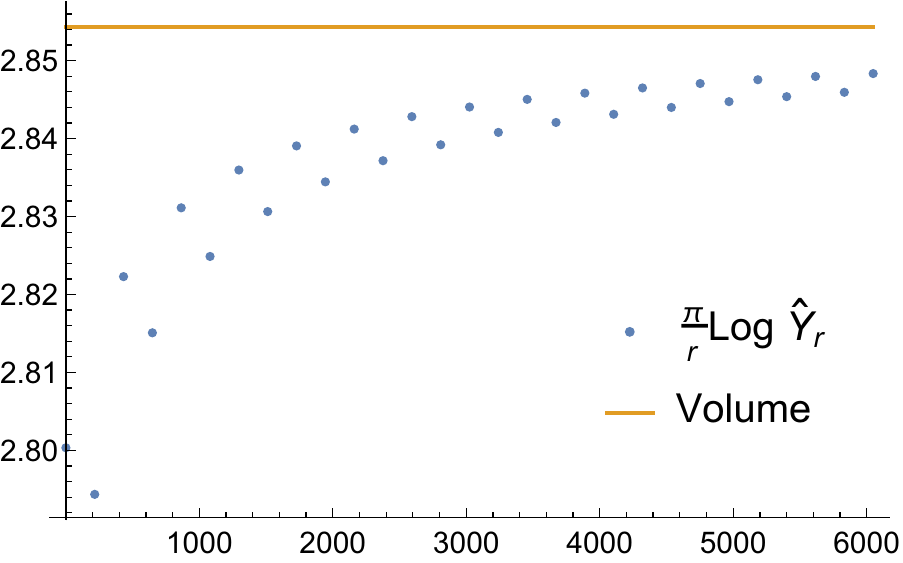}
\end{minipage}
\begin{minipage}{.45\textwidth}
 \begin{flushright}
 {\tabulinesep=1.2mm
 \begin{tabu}{|C{2.3cm}|C{4.5cm}|}
 \hline
  Type & $\left( (1),(23456)\right)$\\  \hline
  Angles & $(0),\left(\frac{\pi}{5},\frac{\pi}{4},\frac{\pi}{4},\frac{\pi}{4},\frac{\pi}{4}\right)$ \\ \hline Edge lengths & $0$ \\  \hline Volume & $2.8543$ \\ \hline $\frac{\pi}{6049}\log \widehat{\mathrm Y}_{6049}$ &  $2.84835$ \\ \hline Error & $0.2\%$\\
  \hline
 \end{tabu}}

\end{flushright}
\end{minipage}

\vspace{1.5cm}

\begin{minipage}{.45\textwidth}
  \includegraphics[width=.95\textwidth]{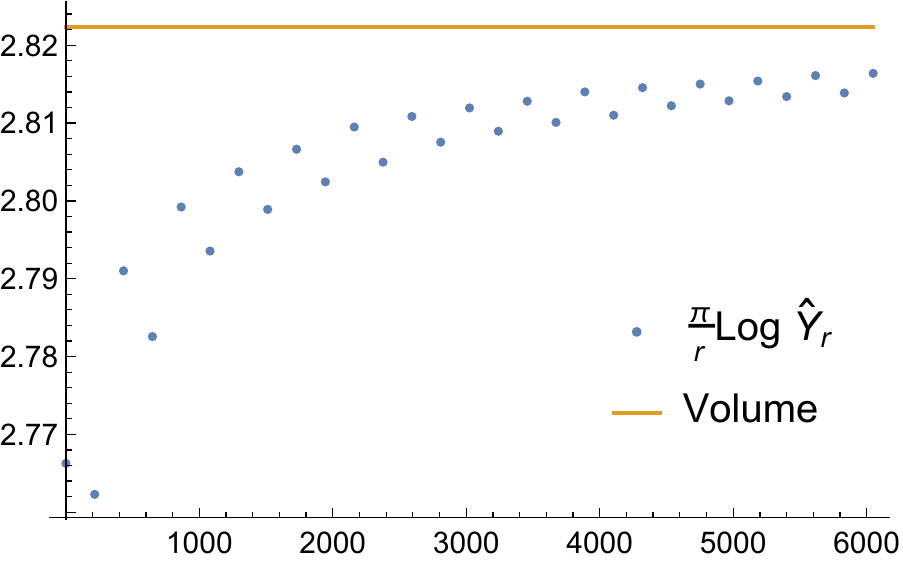}
\end{minipage}
\begin{minipage}{.45\textwidth}
 \begin{flushright}
 {\tabulinesep=1.2mm
 \begin{tabu}{|C{2.3cm}|C{4.5cm}|}
 \hline
  Type & $\left( (1),(23456)\right)$\\  \hline
  Angles & $(0.4005),\left(\frac{\pi}{5},\frac{\pi}{4},\frac{\pi}{4},\frac{\pi}{4},\frac{\pi}{4}\right)$ \\ \hline Edge lengths & $0.3214$ \\  \hline Volume & $2.8223$ \\ \hline $\frac{\pi}{6049}\log \widehat{\mathrm Y}_{6049}$ & $2.8163$ \\ \hline Error & $0.2\%$\\
  \hline
 \end{tabu}}

\end{flushright}
\end{minipage}

\vspace{1.5cm}

\begin{minipage}{.45\textwidth}
  \includegraphics[width=.95\textwidth,left]{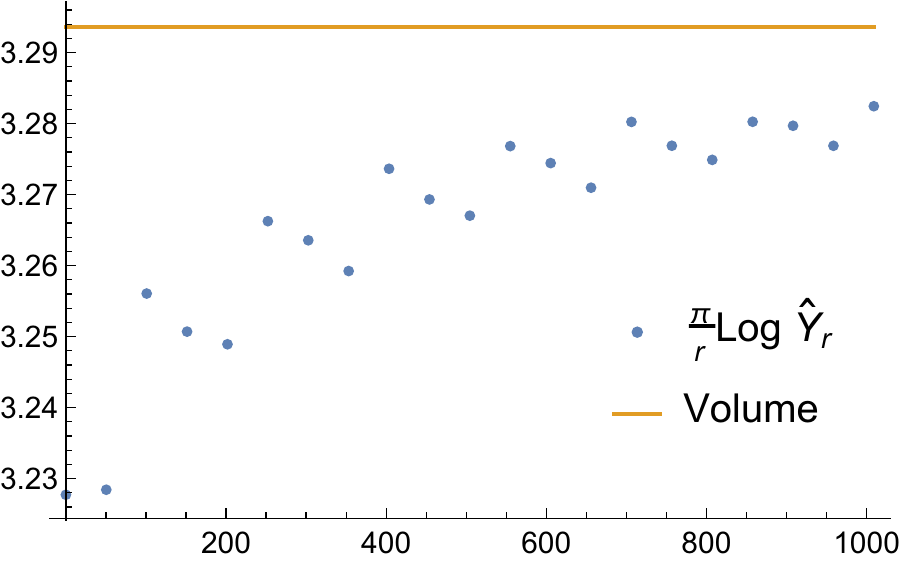}
\end{minipage}
\begin{minipage}{.45\textwidth}
 \begin{flushright}
 {\tabulinesep=1.2mm
 \begin{tabu}{|C{2.3cm}|C{5cm}|}
\hline
  Type & $\left( (12),(3456)\right)$\\  \hline
  Angles & $(0.1638,0.2160),\left(\frac{\pi}{5},\frac{\pi}{6},\frac{\pi}{5},\frac{\pi}{6}\right)$ \\ \hline Edge lengths &  $0.1486, 0.2024$ \\  \hline Volume & $3.2937$ \\ \hline $\frac{\pi}{1009}\log \widehat{\mathrm Y}_{1009}$ & $3.2825$ \\ \hline Error & $0.3\%$\\
  \hline
 \end{tabu}}
\end{flushright}
\end{minipage}

\vspace{1.5cm}

\begin{minipage}{.45\textwidth}
  \includegraphics[width=.95\textwidth,left]{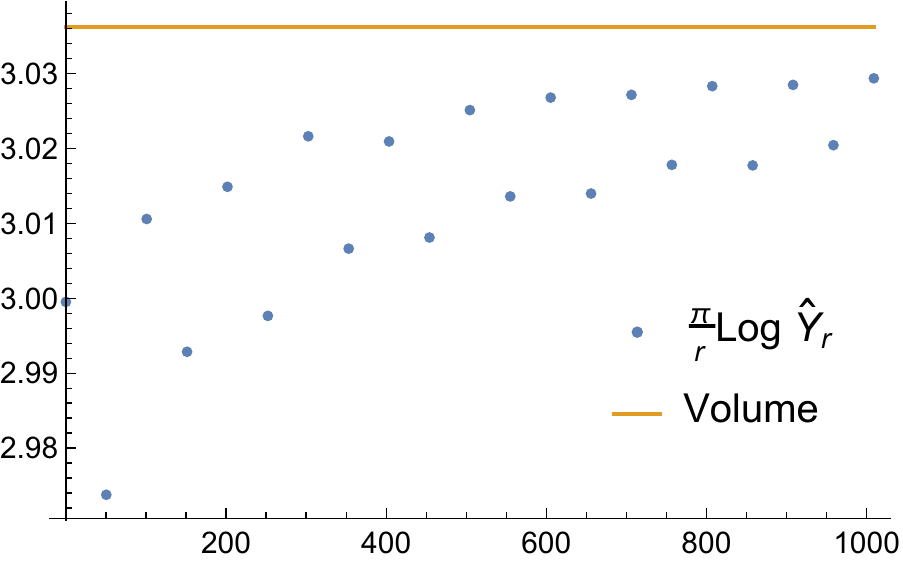}
\end{minipage}
\begin{minipage}{.45\textwidth}
 \begin{flushright}
 {\tabulinesep=1.2mm
 \begin{tabu}{|C{2.3cm}|C{5cm}|}
 \hline
  Type & $\left( (14),(2356)\right)$\\  \hline
  Angles & $(0.3862,0.2302),\left(\frac{\pi}{4},\frac{\pi}{5},\frac{\pi}{4},\frac{\pi}{4}\right)$ \\ \hline Edge lengths & $0.2842, 0.1673$ \\  \hline Volume & $3.0362$ \\ \hline $\frac{\pi}{1009}\log \widehat{\mathrm Y}_{1009}$ & $3.0293$ \\ \hline Error & $0.2\%$\\
  \hline
 \end{tabu}}
\end{flushright}
\end{minipage}

\vspace{1.5cm}

\begin{minipage}{.45\textwidth}
  \includegraphics[width=.95\textwidth,left]{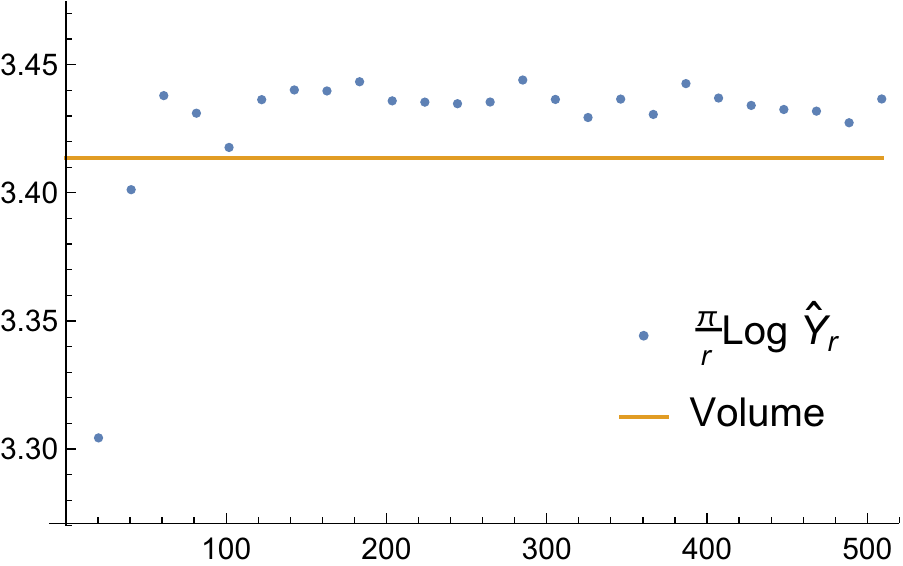}
\end{minipage}
\begin{minipage}{.45\textwidth}
 \begin{flushright}
 {\tabulinesep=1.2mm
 \begin{tabu}{|C{2.3cm}|C{6cm}|}
 \hline
  Type & $\left( (123),(456)\right)$\\  \hline
  Angles & $(0.1282,0.2060,0.2955),\left(\frac{\pi}{8},\frac{\pi}{6},\frac{\pi}{5}\right)$ \\ \hline Edge lengths &  $0.1210, 0.2008,0.29683$  \\  \hline Volume & $3.4136$ \\ \hline $\frac{\pi}{509}\log \widehat{\mathrm Y}_{509}$ & $3.4366$ \\ \hline Error: & $0.7\%$\\
  \hline
 \end{tabu}}
\end{flushright}
\end{minipage}

\vspace{1.5cm}

\begin{minipage}{.45\textwidth}
  \includegraphics[width=.95\textwidth,left]{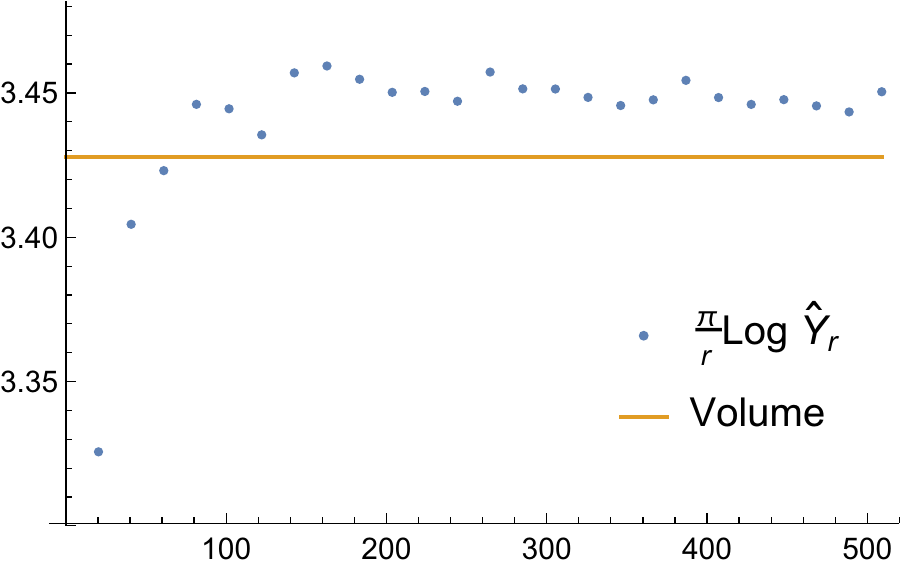}
\end{minipage}
\begin{minipage}{.45\textwidth}
 \begin{flushright}
 {\tabulinesep=1.2mm
 \begin{tabu}{|C{2.3cm}|C{6cm}|}
 \hline
  Type & $\left( (124),(356)\right)$\\  \hline
  Angles & $(0.1042,0.1802,0.1339),\left(\frac{\pi}{7},\frac{\pi}{6},\frac{\pi}{5}\right)$ \\ \hline Edge lengths &  $0.0931, 0.1743,0.1203$ \\  \hline Volume & $3.4277$ \\ \hline $\frac{\pi}{509}\log \widehat{\mathrm Y}_{509}$ & $3.4504$ \\ \hline Error: & $0.7\%$\\
  \hline
 \end{tabu}}
\end{flushright}
\end{minipage}

\vspace{1.5cm}

\begin{minipage}{.45\textwidth}
  \includegraphics[width=.95\textwidth,left]{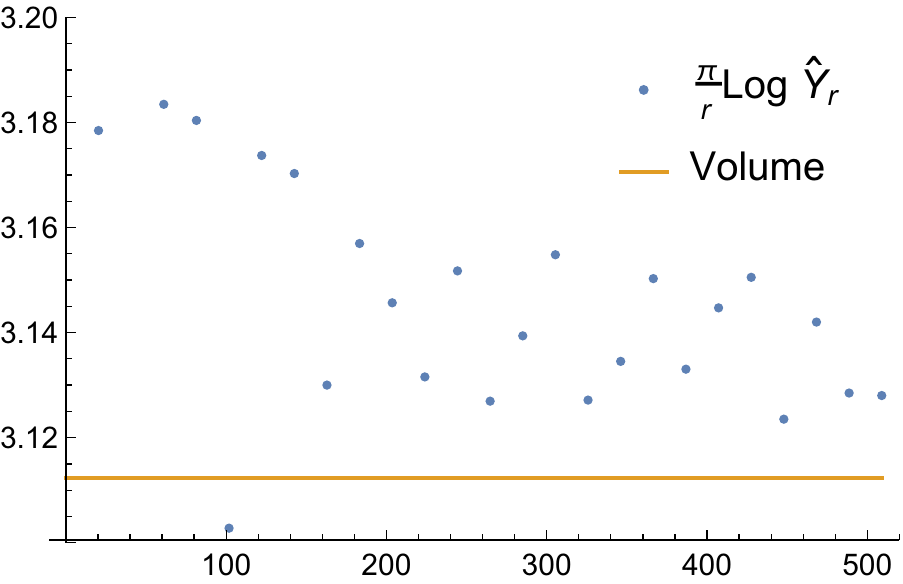}
\end{minipage}
\begin{minipage}{.45\textwidth}
 \begin{flushright}
 {\tabulinesep=1.2mm
 \begin{tabu}{|C{2.3cm}|C{6cm}|}
 \hline
  Type & $\left( (126),(345)\right)$\\  \hline
  Angles & $(0.4041,0.5014,0.4064),\left(\frac{2\pi}{13},\frac{3\pi}{13},\frac{4\pi}{17}\right)$ \\ \hline Edge lengths &  $0.4284, 0.5045,0.3817$  \\  \hline Volume & $3.1123$ \\ \hline $\frac{\pi}{509}\log \widehat{\mathrm Y}_{509}$ & $3.1280$ \\ \hline Error: & $0.5\%$\\
  \hline
 \end{tabu}}
\end{flushright}
\end{minipage}

 \noindent 
Giulio Belletti\\
Scuola Normale Superiore\\
Pisa, Italy\\ 
(giulio.belletti@sns.it)
\\

\noindent
Tian Yang\\
Department of Mathematics\\  Texas A\&M University\\
College Station, TX 77843, USA\\
(tianyang@math.tamu.edu)

\end{document}